\documentclass[11pt]{article}
\usepackage{amsmath} % \font used for R in Real numbers
\usepackage{verbatim}
\usepackage{amssymb, float}
\usepackage{color}
\usepackage{graphicx,epsfig}
\graphicspath{{figures/}}
\usepackage[applemac]{inputenc}
\title{Equivalence between Differential Inclusions Involving Prox-regular sets and maximal monotone operators}
\newtheorem{theorem}{Theorem}[section]
\newtheorem{corollary}[theorem]{Corollary}
\newtheorem{lmm}[theorem]{Lemma}
\newtheorem{prpstn}[theorem]{Proposition}
\newtheorem{dfntn}[theorem]{Definition}
\newtheorem{xmpl}[theorem]{Example}
\newtheorem{remark}[theorem]{Remark}
\def\beq{\begin{equation}}
\def\eeq{\end{equation}}
\def\baq{\begin{eqnarray}}
\def\eaq{\end{eqnarray}}
\def\baqn{\begin{eqnarray*}}
\def\eaqn{\end{eqnarray*}}

\def\image #1 (#2,#3) (echelle #4) #5{
\dimen2=#2
\dimen3=#3
\divide \dimen2 by 1000
\multiply \dimen2 by #4
\divide \dimen3 by 1000
\multiply \dimen3 by #4
\setbox1 =\vbox to \dimen2{\hsize=\dimen3\vfill\special{picture #1
scaled #4}}
\vbox{\hsize=\dimen3\box1\medskip\centerline{#5}}
}
\begin{document}
\author{
Samir {\sc Adly}\thanks{Laboratoire XLIM, Universit\'e de Limoges, France. Email: samir.adly@unilim.fr.}, Abderrahim \textsc{Hantoute}\thanks{Center for Mathematical Modeling (CMM), Universidad de Chile. Email: ahantoute@dim.uchile.cl.}, Bao Tran \textsc{Nguyen}\thanks{Universidad de O'Higgins, Rancagua, Chile. Email: nguyenbaotran31@gmail.com }
}

\maketitle
\begin{abstract}\noindent In this paper, we study the existence and the stability in the
sense of Lyapunov of solutions for\ differential inclusions
governed by the normal cone to a prox-regular set and subject to a
Lipschitzian perturbation. We prove that such, apparently, more
general nonsmooth dynamics can be indeed remodelled into the
classical theory of differential inclusions involving maximal
monotone operators. This  result is new in the literature and
permits us to make use of the rich and abundant achievements in
this class of monotone operators to derive the desired existence
result and stability analysis, as well as the continuity and
differentiability properties of the solutions. This going back and
forth between these two models of differential inclusions is made
possible thanks to a viability result for maximal monotone
operators. As an application, we study a Luenberger-like observer,
which is shown to converge exponentially to the actual state when
the initial value of the state's estimation remains in a
neighborhood of the initial value of the original system.\end{abstract}

%\tableofcontents[hideallsubsections]
\setcounter{tocdepth}{1}
\tableofcontents

%\section{Introduction, Motivation and Discussions}

\section{Introduction}
We consider in this paper the existence and the stability in the
sense of Lyapunov of solutions  of\ the following differential
inclusion, given in a Hilbert space $H,$
\begin{equation}
\begin{cases}
\dot{x}(t)\in f(x(t))-\mathrm{N}_{C}(x(t)) & \,\text{for almost every }\,t\geq0, \\
x(0;x_{0})=x_{0}\in C, &
\end{cases}
\label{101}
\end{equation}
where $\mathrm{N}_{C}$ is the normal cone to an $r$-uniformly
prox-regular closed subset $C$ of $H.$ The dynamical system driven
by the set $C$ is subject to a Lipschitz continuous perturbation
mapping $f$ defined on $H.$ For a given initial condition $x_0\in
C$, by a solution of $(\ref{101})$ we mean an absolutely
continuous function $x(\cdot;x_{0}):[0,+%
\infty)\rightarrow H,$\ with\ $x(0;x_{0})=x_{0},$ which satisfies $(\ref{101})$ for
almost every (a.e.) $t\geq0;$ hence, in particular, $x(t)\in C$ for all $%
t\geq0$ (since the normal cone is empty outside the set $C$).
Indeed, such a solution is necessarily Lipschitz continuous on
each
interval of the form $[0,T]$ for $T\ge 0$ (see Theorem \ref{theo3.1}). Differential inclusion (\ref%
{101}) appears in the modeling of many concrete problems in economics,
unilateral mechanics, electrical engineering as well as optimal control (see
eg. \cite{abb}, \cite{brogli4}, \cite{CP}, \cite{MT2}, \cite{stewart} and references therein.)

It was recently shown in \cite{MT1} and \cite{MT2} that $(\ref{101})$ has
one and only one (absolutely continuous) solution, which satisfies the
imposed initial condition. These authors employed a regularization approach
based on the Moreau-Yosida approximation, and use the nice properties of
uniform prox-regularity to show that the approximate scheme converges to the
required solution. In this way, such an approach repeats those\ arguments of
approximation ideas which, previously, were extensively used in the setting
of differential inclusions with maximal monotone operators.

Problems dealing with the stability of solutions of $(\ref{101})$, namely
the characterization of weakly lower semi-continuous Lyapunov pairs and
functions, have been developed in \cite{MH} following the same strategy,
also based on Moreau-Yosida approximations. Most of works on these problems
use indeed this natural approximation approach; see, e.g. \cite{MT1,MT2,MH}.

In this paper, at a first glance we provide a different, but quite direct,
approach to tackle this problem. We prove that problem $(\ref{101})$ can be
equivalently written as a differential inclusion given in the current
Hilbert setting under the form
\begin{equation}
\begin{cases}
\dot{x}(t)\in g(x(t))-A(x(t)) & \,\text{a.e. \, }\,t\in\lbrack0,T], \\
x(0;x_{0})=x_{0}\in\text{dom}A, &
\end{cases}
\label{Maxmoninclusion}
\end{equation}
where $A:H \rightrightarrows H$ is an appropriate (depending on
$C$) maximal monotone operator defined on $H,$ and $g:H\to H$ is a
Lipschitz continuous mapping. Then, it will be sufficient to apply
the classical theory of maximal monotone operators (\cite{B}; see,
also, \cite{AHT1, AHT2}) to analyze the existence and the
stability of solutions for differential inclusion $(\ref{101})$.
The concept of invariant sets will be the key tool to go back and
forth between inclusions $(\ref{101})$ and
(\ref{Maxmoninclusion}). Invariant sets with respect to
differential inclusions governed by maximal monotone operators
have been studied and characterized in \cite{AHT1, AHT2}. Other
references for invariant sets, also referred to as viable sets,
and the related theory of Lyapunov stability are \cite{AC, CM, KS,
P} among others. We also refer to \cite{DRW} for an interesting
criterion for weakly invariant sets, which is established in the
finite-dimensional setting for differential inclusions governed by
one-sided Lipschitz multivalued mappings with nonempty convex and
compact values. This result has been used in \cite{CP}, always in
finite dimensions, to provide weakly and strongly invariance
criteria for closed sets with respect to more general differential
inclusions where the set $C$ in (\ref{101}) is time dependent and
$f$ is a Lipschitzian multivalued mapping.

We shall also provide different criteria for the so-called $a$-Lyapunov pairs of lower
semi-continuous functions to extend some of the results given in \cite{AHT1,
AHT2, MH} to the current setting. It is worth to observe that the assumption
of uniformly prox-regularity is required to obtain global solutions of $(\ref%
{101})$, which are defined on the whole interval $[0,T].$ However, our
analysis also works in the same way when the set $C$ is prox-regular at $%
x_{0}$ rather than being a uniformly prox-regular set; but, in this case, we
only obtain a local solution defined around $x_0.$

This paper is organized as follows. After giving the necessary notations and
preliminary results in Section 2, we review and study in Section 3 different
aspects of the theory of differential inclusions governed by maximal
monotone operators, including the existence of solutions, and we provide a
stability results dealing with the invariance of closed sets with respect to
such differential inclusions. In Sections 4, we provide the new proof of the
existence of solutions for differential inclusions \ involving normal cones
to $r$-uniformly prox-regular sets. Section 5 is devoted to the
characterization of lower semi-continuous $a$-Lyapunov pairs and functions.
Inspired from the recent paper \cite{TBP}, we give in section 6 an
application of our result to a Luenberger-like observer.
\section{Preliminaries and examples}

\subsection{Preliminary results}

In this paper, $H$ is a Hilbert space endowed with an inner product $%
\langle\cdot,\cdot\rangle$ and an associated norm $||\cdot||.$ The strong
and weak convergences in $H$ are denoted by $\rightarrow$ and $%
\rightharpoonup,$ resp. We denote by $\mathrm{B}(x,\rho)$ the closed ball centered at
$x\in H$ of radius $\rho>0,$ and particularly we use $\mathbb{B}$ for the
closed unit ball. The null vector in $H$ is written $0$.
Given a set $S\subset H,$ by $\text{co}S,$ $\text{cone}S$ and $\,\overline{S}
$ we respectively denote the \textit{convex hull, the conic hull} and the
\textit{closure} of $S.$ The dual cone of $S$ is the set
\begin{equation*}
S^{\ast}:=\{x^{\ast}\in H\mid\left\langle x^{\ast},x\right\rangle \leq0\text{
for all }x\in S\}.
\end{equation*}
The \textit{indicator }and the \textit{distance functions} are respectively
given by
\begin{equation*}
\mathrm{I}_{S}(x):=0\,\,\,\text{if}\,\,x\in S;\,+\infty\,\,\text{otherwise},%
\,\,\,\,\,\,d_{S}(x):=\inf\{||x-y||:\,\,y\in S\}
\end{equation*}
(in the sequel we shall adopt the convention $\inf\nolimits_{\emptyset
}=+\infty$). We shall write $\overset{S}{\rightharpoonup}$ for the
convergence when restricted to the set $S.$ We denote $%
\Pi_{S}$ the (\textit{orthogonal) } \textit{projection
mapping} onto $S$ defined as
\begin{equation*}
\Pi_{S}(x):=\{y\in S:\,\,||x-y||= d_{S}(x)\}.
\end{equation*}
 It is known that $\Pi_{S}$ is nonempty-valued on a dense subset of $H\setminus S$ (see e.g. \cite%
{C}).

For an extended real-valued function $\varphi:H\rightarrow\overline{\mathbb{R%
}},$ we denote $\text{dom}\varphi:=\{x\in H\mid\varphi(x)<+\infty\}$ and $%
\text{epi}\varphi:=\{(x,\alpha)\in H\times\mathbb{R}\,|\,\,\varphi
(x)\leq\alpha\}.$ Function $\varphi$ is lower semi-continuous (lsc) if $%
\text{epi}\varphi$ is closed. The \textit{contingent directional derivative}
of $\varphi$ at $x\in\text{dom}\varphi$ in the direction $v\in H$ is
\begin{equation*}
\varphi^{\prime}(x,v):=\underset{t\rightarrow0^{+},w\rightarrow v}{\lim\inf }%
\frac{\varphi(x+tw)-\varphi(x)}{t}.
\end{equation*}
A vector $\xi\in H$ is called a \textit{proximal subgradient }of $\varphi$
at $x\in H$, written $\xi\in\partial_{P}\varphi(x),$ if there are $\rho>0$
and $\sigma\geq0$ such that
\begin{equation*}
\varphi(y)\geq\varphi(x)+\langle\xi,y-x\rangle-\sigma||y-x||^{2},\,\,\,%
\forall\,\,y\in B_{\rho}(x);
\end{equation*}
a \textit{Fr\'echet subgradient} of $\varphi$ at $x$, written $\xi\in
\partial_{F}\varphi(x),$ if
\begin{equation*}
\varphi(y)\geq\varphi(x)+\langle\xi,y-x\rangle+o(||y-x||),\,\,\forall\,\,y%
\in H;
\end{equation*}
and a \textit{basic (or Limiting) subdifferential} of $\varphi$ at $x,$
written $\xi\in\partial_{L}\varphi(x),$ if there exist sequences $%
(x_{k})_{k} $ and $(\xi_{k})_{k}$ such that
\begin{equation*}
x_{k}\overset{\varphi}{\rightarrow}x,\,\,(\text{i.e., }x_{k}\rightarrow x%
\text{ and }\varphi(x_{k})\rightarrow\varphi(x)),\,\,\xi_{k}\in\partial
_{P}\varphi(x_{k}),\,\,\xi_{k}\rightharpoonup\xi.
\end{equation*}
If $x\notin\text{dom}\varphi,$ we write $\partial_{P}\varphi
(x)=\partial_{F}\varphi(x)=\partial_{L}\varphi(x)=\emptyset.$ In particular,
if $S$ is a closed set and $s\in S,$ we define the \textit{proximal normal
cone} to $S$ at $s$ as $\mathrm{N}_{S}^{P}(s)=\partial_{P}\mathrm{I}_{S}(s),$
the \textit{Fr\'echet normal} to $S$ at $s$ as\ $\mathrm{N}%
_{S}^{F}(s)=\partial_{F}\mathrm{I}_{S}(s),$ the \textit{limiting
normal cone} to $S$ at $s$ as
$\mathrm{N}_{S}^{L}(s)=\partial_{L}\mathrm{I}_{S}(s),$ and
\textit{the Clarke\ normal cone} to $S$ at $s$ as $\mathrm{N}_{S}^{C}(s)=%
\overline{\text{co}}\mathrm{N}_{S}^{L}(s).$ Equivalently, we have that $%
\mathrm{N}_{S}^{P}(s)=\text{cone}(\Pi_{S}^{-1}(s)-s),$ where $%
\Pi_{S}^{-1}(s):=\{x\in H\mid s\in\Pi_{S}(x)\}.$ The Bouligand and weak
\textit{Bouligand tangent cones to }$S$ at $x$ are defined as
\begin{align*}
T_{S}^{B}(x) & :=\left\{ v\in H\,|\,\exists\,\,x_{k}\in S,\exists
\,\,t_{k}\rightarrow0,\,\,\text{st}\,\,t_{k}^{-1}(x_{k}-x)\rightarrow v%
\mbox{ as } k\to+\infty\right\} \\
T_{S}^{\text{w}}(x) & :=\left\{ v\in H\,|\,\exists\,\,x_{k}\in
S,\exists\,\,t_{k}\rightarrow0,\,\,\text{st}\,\,t_{k}^{-1}(x_{k}-x)%
\rightharpoonup v \mbox{ as } k\to+\infty\right\} ,\text{ resp.}
\end{align*}
We also define the \textit{Clarke subgradient} of $\varphi$ at $x,$ written $%
\partial_{C}\varphi(x)$,$\ $as the vectors $\xi\in H$ such that $(\xi ,-1)\in%
\mathrm{N}_{\text{epi}\varphi}^{C}(x,\varphi(x))$, and the \textit{singular
subgradient} of $\varphi$ at $x,$ written $\partial_{\infty }\varphi(x),\ $%
as the vectors $\xi\in H$ such that $(\xi,0)\in\mathrm{N}_{\text{epi}%
\varphi}^{P}(x,\varphi(x));$ in particular, if $\xi
\in\partial_{\infty}\varphi(x),$ then\ there are sequences $x_{k}\overset{%
\varphi}{\rightarrow}x,\,\xi_{k}\in\partial_{P}\varphi(x_{k}),\,$and $%
\lambda_{k}\rightarrow0^{+}$ such that $\lambda_{k}\xi_{k}\rightarrow\xi.$
Observe that $\partial_{P}\varphi(x)\subset\partial_{F}\varphi(x)\subset
\partial_{L}\varphi(x)\subset\partial_{C}\varphi(x).$ For all these concepts
and their properties we refer to the book \cite{M}.

We shall frequently use the following version of Gronwall's Lemma:

\begin{lmm}
\label{l2}\rm{(Gronwall's Lemma; see, e.g., \cite{AB})} Let $T>0\ $and
$a,b\in L^{1}(t_{0},t_{0}+T;\mathbb{R})$ such that\ $b(t)\geq0$ a.e.
$t\in\lbrack t_{0},t_{0}+T].$ If, for some $0\leq\alpha<1$, an absolutely
continuous function $w:[t_{0},t_{0}+T]\rightarrow\mathbb{R}_{+}$ satisfies
\[
(1-\alpha)w^{\prime}(t)\leq a(t)w(t)+b(t)w^{\alpha}(t)\,\,\ \text{a.e.}%
\,\,t\in\lbrack t_{0},t_{0}+T],
\]
then
\[
w^{1-\alpha}(t)\leq w^{1-\alpha}(t_{0})e^{\int_{t_{0}}^{t}a(\tau)d\tau}%
+\int_{t_{0}}^{t}e^{\int_{s}^{t}a(\tau)d\tau}b(s)ds,\,\,\forall t\in\lbrack
t_{0},t_{0}+T].
\]

\end{lmm}

\subsection{Some examples}

\begin{xmpl}\normalfont
(\textit{Parabolic Variational Inequalities}). Let $\Omega\subset\mathbb{R}%
^N $ be an open bounded subset with a smooth boundary $\partial \Omega$. Let
us consider the following boundary value problem, with Signorini conditions,
of finding a function $(t,x)\mapsto u=u(t,x)$ such that
\begin{equation*}
(P)\left\{
\begin{array}{l}
\frac{\partial u}{\partial t}-\Delta u=f,\;\;(t,x)\in [0,T]\times\Omega, \\
u(0,x)=u_0(x),\;x\in\Omega \mbox{ (initial condition)} \\
u\geq 0,\;\frac{\partial u}{\partial n}\geq 0 \mbox{ and } u\frac{\partial u%
}{\partial n}=0\mbox{ for } (t,x)\in [0,T]\times \partial \Omega.%
\end{array}
\right.
\end{equation*}
It is well-known that the weak formulation of problem $(P)$ is given by the
following parabolic variational inequalities
\begin{equation*}
(VI)\left\{
\begin{array}{l}
\mbox{Find } u\in {\mathcal{C}} \mbox{ such that } \\
\displaystyle\int_\Omega u^{\prime }(t)(v(t)-u(t))dx+\int_\Omega \nabla
u(t)\cdot\nabla (v(t)-u(t))dx\geq \\
\qquad \qquad \displaystyle \int_\Omega f(t)(v(t)-u(t))dx,\;\forall v\in {%
\mathcal{C}}, \mbox{ a.e. } t\in [0,T].%
\end{array}
\right.
\end{equation*}
Here, ${\mathcal{C}}=\{v\in L^2(0,T;H^1(\Omega)) \,:\, v(t)\in C
\mbox{ for
a.e. } t\in [0,T]\},$ where $C=\{v\in H^1(\Omega)\,:\, v\geq 0 \mbox{ on }
\partial \Omega\}$. It is easy to see that the parabolic variational
inequality (VI) is of the form \textrm{(\ref{101})}. The convexity structure
of the set $C$ (since it is a closed convex cone) makes the problem (VI)
standard and may be straightforward. Let us consider now a function $g:%
\mathbb{R}\to \mathbb{R}$ and define the new set $C$ with the associated set
${\mathcal{C}} $
\begin{equation*}
C=\{v\in H^1(\Omega)\,:\, g(v(x))\geq 0 \mbox{ for } x\in\partial \Omega\}.
\end{equation*}
The set $C$ is no more convex and some sufficient conditions on the function
$g$ are necessary to ensure the prox-regularity of the sets $C$ and ${%
\mathcal{C}}$ (see \textrm{\cite{ant}} for more details).
\end{xmpl}

\begin{xmpl}\normalfont
(\textit{Nonlinear Differential Complementarity Systems}). Let us consider
the following ordinary differential equation, coupled with a complementarity
condition,
\begin{equation*}
(NDCS)\left\{
\begin{array}{l}
\dot{x}(t)= f(x(t))+\lambda (t),\;t\in [0,T] \\
\lambda(t),  g(x(t))\geq 0,\, \langle \lambda(t), g(x(t)) \rangle = 0,
\end{array}
\right.
\end{equation*}
where $f:\mathbb{R}^n\to\mathbb{R}^n$, $g:\mathbb{R}^n\to\mathbb{R}^m$ are of
class $C^1$ and $\lambda: [0,T]\to\mathbb{R}^m$ is a Lagrange multiplier
(unknown function). We have that
\begin{equation*}
\lambda(t),  g(x(t))\geq 0,\, \langle \lambda(t), g(x(t)) \rangle = 0 \Longleftrightarrow -\lambda(t)\in
\mathrm{N}_{\mathbb{R}^m_+}(g(x(t))).
\end{equation*}
Hence, (NDCS) is written as
\begin{equation*}
\dot{x}(t) \in f(x(t))-\mathrm{N}_{\mathbb{R}^m_+}(g(x(t))),
\end{equation*}
with $\mathrm{N}_{\mathbb{R}^m_+}(g(x(t)))=\partial \mathrm{I}_{\mathbb{R%
}^m_+}(g(x(t))$, where $\partial$ denotes the subdifferential in the sense of convex analysis. If we suppose a qualification condition such as, e.g., $%
\nabla g$ is surjective, then, using classical chain rules for Clarke generalized subdifferential (see e.g. \textrm{%
\cite{rw}}), we get
\begin{equation*}
\partial (\mathrm{I}_{\mathbb{R}^m_+}\circ g)(x)=\nabla g(x)^T \mathrm{N}_{%
\mathbb{R}^m_+}(g(x)).
\end{equation*}
By setting $C=\{x\in\mathbb{R}^n\,:\, g(x)\geq 0\}$, it is easy to see that
problem (NCDS) is equivalent to the following differential inclusion
\begin{equation*}
\dot{x}(t)\in f(x(t))-\mathrm{N}_C(x(t)),
\end{equation*}
which is of the form of $\textrm{(\ref{101})}$. Under some sufficient
conditions on the vectorial function $g$ (see \textrm{\cite[Theorem 3.5]{ant}%
}), we show that the set $C$ is $r$-prox-regular.
\end{xmpl}

\noindent Many problems in power converters electronics and unilateral
mechanics can be modeled by nonlinear differential complementarity problems
of the form (NDCS) (see e.g. \cite{abb} and \cite{stewart})).

\section{Differential inclusions involving maximal monotone operators}

We review in this section some aspects of the theory of differential
inclusions involving maximal monotone operators. Namely, we provide an
invariance result for associated closed sets that we use in the sequel.

Given a set-valued operator $A:H \rightrightarrows H,$ which we
identify with its graph, we denote its \textit{domain} by
$\text{dom}A:=\{x\in H\mid A(x)\neq\emptyset\}.$ The operator\ $A$
is \textit{monotone} if
\begin{equation*}
\left\langle x_{1}-x_{2},y_{1}-y_{2}\right\rangle \geq0\text{ \ \ for all \ }%
(x_{1},y_{1}),\text{ }(x_{2},y_{2})\in A,
\end{equation*}
and $\alpha$\textit{-hypomonotone} for $\alpha\geq0$ if the operator $%
A+\alpha\,\text{id}$ is monotone, where $\text{id}$ is the identity mapping.
We say that\ $A$ is \textit{maximal monotone} if $A$ is monotone and
coincides with every monotone operator containing its graph. In such a case,
it is known that $A(x)$ is convex and closed for every $x\in H.$ We shall
denote by $(A(x))^{\circ},$ $x\in\text{dom}A,$ the set of minimal norm vectors\
in $A(x);$ i.e., $(A(x))^{\circ}:=\{y\in A(x)\,|\,||y||=\text{min}_{z\in
A(x)}||z||\}; $ hence, for any vector $x\in\text{dom}A$ and $y\in H,$ the
set $\Pi_{A(x)}(y)$ is a singleton and we have that $(y-A(x))^{\circ}=y-\Pi
_{A(x)}(y).$

We consider the following differential inclusion
\begin{equation}
\dot{x}(t)\in f(x(t))-A(x(t)),\text{ \ }t\in\lbrack0,\infty),\text{ }%
x(0;x_0)=x_{0}\in{\overline{\text{dom}A}},  \label{201}
\end{equation}
governed by a maximal monotone operator $A: H \rightrightarrows H$, which is
subject to a perturbation by a $(\kappa-)$Lipschitz continuous mapping $%
f:H\to H.$ By a strong solution of (\ref{201}) starting at $x_0\in{\overline{%
\text{dom}A}}$ we refer to an absolutely continuous function $x(\cdot;x_0)$
which satisfies (\ref{201}) for a.e.\,$t\ge 0$, together with the initial
condition $x(0;x_0)=x_{0}$. It is known that (\ref{201}) processes a unique
strong solution whenever $x_0\in{\text{dom}A}$, $H$ is finite-dimensional, $%
\text{int(dom$A$)}\ne \emptyset$, or $A$ is the subdifferential of convex,
proper, and lower semicontinuous function. More generally, we call $%
x(\cdot;x_0)$ a weak solution of (\ref{201}) starting at $x_0\in{\overline{%
\text{dom}A}}$, the unique continuous function which is the uniform limit of
strong solutions $x(\cdot;x_k)$ with $(x_k)\subset {\text{dom}A}$ converging
to $x_0$.

The following result provides other properties of the solutions of (\ref{201}%
); for more details we refer to the book \cite{B}. To denote the
right-derivative whenever it exists we use the notation
\begin{equation*}
\frac{d^{+}x(t;x_{0})}{dt}:= \underset{h \downarrow 0}{\lim}\frac{%
x(t+h;x_0)- x(t)}{h}.
\end{equation*}

\begin{prpstn}
\label{Brezis}
Fix  $x_0, y_0 \in \overline{\normalfont{\text{dom}}A}$. Then system {\rm (\ref{201})} has a unique
continuous solution $x(t)\equiv x(t;x_0)$, $t\ge 0$,  such that, for all$\ s,t\geq0$
$$
x(s;x(t;x_{0}))=x(t+s;x_{0}), \quad
\left\Vert x(t;x_{0})-x(t;y_{0})\right\Vert \leq e^{\kappa t}||x_{0}-y_{0}||.
$$
Moreover, if $x_0 \in \normalfont{\text{dom}}A$, then
$$
\frac{d^{+}x(t;x_{0})}{dt}=[f(x(t;x_{0}))-A(x(t;x_{0}))]^{\circ}=f(x(t;x_{0}))-\Pi_{A(x(t;x_{0}))}(f(x(t;x_{0}))),
$$
and the function $t\rightarrow\frac{d^{+}x(t)}{dt}$ is
right-continuous at every $t\geq0$ with
\begin{equation}
\left\Vert \frac{d^{+}x(t)}{dt}\right\Vert \leq e^{\kappa t}\left\Vert \frac {%
d^{+}x(0)}{dt}\right\Vert .  \label{prop2.2}
\end{equation}
\end{prpstn}

We are going to characterize those closed sets which are invariant with
respect to differential inclusion (\ref{201}).

\begin{dfntn}
\label{invariant}A closed set $S\subset H$ is strongly invariant for ${\rm (\ref{201}%
)}$ if every  solution  of ${\rm(\ref{201})}$ starting in $S$ remains in this
set for all time $t\geq0.$ \\
The set $S\subset H$ is weakly invariant for ${\rm (\ref{201}%
)}$ if for every $x_0\in S$,  there exists a solution  $x(\cdot;x_0)$ of ${\rm(\ref{201})}$ such that   $x(t;x_0)\in S$ for all time $t\geq0.$ \\
When differential inclusion ${\rm (\ref{201})}$ has a  unique solution for every given initial condition, both notions coincide, and we simply say in this case that $S$ is invariant.
\end{dfntn}

Due to the semigroup property in Proposition \ref{Brezis}, it is immediately
seen that $S$ is invariant iff every solution of $(\ref{201})$ starting in $%
S $ remains in this set for all sufficiently small time $t\geq0.$ The issue
with these sets, also referred to as \textit{viable sets} for (\ref{201});
see, \cite{AC}, is to find good characterizations via explicit criteria,
which do not require an a-priori computation of the solution of (\ref{201}).
An extensive research has been done to solve this problem for different
kinds of differential inclusions and equations (\cite{C, CLR}). Complete
primal and dual characterizations are given in \cite{AHT1, AHT2}.
\begin{prpstn}\label{co.2}
 Assume that $A$ is a monotone operator, and let $S$ be a closed subset of $\normalfont{\text{dom}}A$. Suppose that $x(\cdot)$ is an absolutely continuous function such that $$
\dot{x}(t)\in f(x(t))-A(x(t)) \,\,\text{a.e.}\,\,\, t\in\left[ 0,T\right], x(0) =x_0\in S.
$$
If there are some numbers $m,\rho>0$ such that
\begin{equation}\label{2.0.0}
\underset{\xi\in\mathrm{N}_{S}^{P}(x)}{\ \sup}\;\;\underset{x^{\ast}\in A(x)\cap
\mathrm{B}(0, m)}{\min}\langle\xi,f(x)-x^{\ast}\rangle\leq 0 \,\,\, \forall x\in S\cap
\mathrm{B}(x_{0},\rho),
\end{equation}
then there is some $T^{\ast}\in(0,T]$ such that $x(t)\in S$ for all $%
t\in\lbrack0,T^{\ast}].$
\end{prpstn}
\begin{proof}
According to \cite{B}, there exists a maximal monotone operator
$\hat{A}$ which extends the monotone operator $A$, so that
$x(\cdot)$ is the unique solution of the differential inclusion
$$
\dot{x}(t) \in f(x(t)) - \hat{A}(x(t)) \,\,t \ge 0, x(0)= x_0 \in
S.
$$ By (\ref{2.0.0}) and the Lipschitz continuity of $f$,
 we choose $k>m$ such that for all $x \in \mathrm{B}(x_0, \rho) \cap S$ one has
\begin{equation}\label{2.0.1}
\Vert \Pi_{A(x)}(f(x))\Vert \le k,\,\, \Vert (f(x)- A(x))^{\circ} \Vert \le k,\,\, \hat{A}(x) \cap \mathrm{B}(0, k) \ne \emptyset.
\end{equation}
Hence, for
$$
S_{k}:= \{x \in S \cap \text{dom}\hat{A} \,\mid \,\Vert (f(x)-
A(x))^{\circ}\Vert \le k\},
$$ it holds
$ S_k \cap \mathrm{B}(x_0, \rho) = S \cap \mathrm{B}(x_0, \rho), $
and so, (\ref{2.0.0}) implies that, for every $x \in S \cap
\mathrm{B}(x_0, \frac{\rho}{2}),$
$$
\underset{\xi\in\mathrm{N}_{S_k}^{P}(x)}{\
\sup}\;\;\underset{x^{\ast}\in \hat{A}(x)\cap
\mathrm{B}(0,k)}{\min}\langle\xi,f(x)-x^{\ast}\rangle \leq
\underset{\xi\in\mathrm{N}_{S}^{P}(x)}{\
\sup}\;\;\underset{x^{\ast}\in A(x)\cap \mathrm{B}(0,
k)}{\min}\langle\xi,f(x)-x^{\ast}\rangle \leq 0.
$$
Consequently, the conclusion follows by by applying
\cite[Corollary 5]{AHB}.
\end{proof}
\section{The existence result}

In this section, we use tools from convex and variational analysis
to prove the existence and uniqueness of a solution for the
differential inclusion $(\ref{101}),$
\begin{equation*}
\begin{cases}
\dot{x}(t)\in f(x(t))-\mathrm{N}_{C}(x(t)) & \,\text{a.e. }\,t\geq0 \\
x(0,x_{0})=x_{0}\in C, &
\end{cases}%
\end{equation*}
where $\mathrm{N}_{C}$ is the proximal, or, equivalently, the limiting,
normal cone to an $r$-uniformly prox-regular closed subset $C$ of $H$, and\ $%
f$ is a Lipschitz continuous mapping. We shall denote by $x(\cdot;x_{0})$
the solution of this inclusion.
\begin{dfntn}
\label{defi21}{\rm (see \cite{MT1,PRT})}\ For positive numbers
$r$ and $\alpha,$ a closed set $S$ is said to be $(r,\alpha)$-prox-regular at
$\overline{x}\in S$ provided that one has
$
x=\Pi_{S}(x+v),
$
for all $x\in S\cap \mathrm{B}(\overline{x},\alpha)$
and all $v\in\mathrm{N}_{S}^{P}(x)$ such that $||v||<r.$
\\
The set $S$ is $r$-prox-regular (resp., prox-regular) at $\overline{x}$ when
it is $(r,\alpha)$-prox-regular at $\overline{x}$ for some real $\alpha>0$
(resp., for some numbers $r, \alpha>0$). The set $S$ is said to be
$r$-uniformly prox-regular when $\alpha=+\infty.$
\end{dfntn}
It is well-known and easy to check that when $S$ is $r$-uniformly
prox-regular, then for every $x\in S,$ $\mathrm{N}_{S}^{P}(x)=\mathrm{N}%
_{S}^{C}(x)$; thus, for such sets we will simply write $\mathrm{N}_{S}(x)$
to refer to each one of these cones, and write $T_S(x)$ to refer to the
Bouligand tangent cone $T^B_S(x)=(\mathrm{N}_{S}(x))^*$.

We have the following property of $r$-uniformly prox-regular sets,
(see e.g. \cite{CT, MT1, MH, PRT}).

\begin{prpstn}
\label{prop2.1} Let $S$ be a closed subset of $H.$ If $S$ is $r$-uniformly
prox-regular, then the set-valued mapping defined by $x\mapsto \mathrm{N}%
_{S}^{P}(x)\cap\mathbb{B}$ is $\frac{1}{r}$-hypomonotone.
\end{prpstn}

Before we state the main theorem of this section we give a useful
characterization of prox-regularity.

\begin{lmm}\label{extension}The following statements are equivalent for every closed set $C\subset
H$ and every $m>0,$

{\rm{(a)}}\ $C$ is $r^{\prime}$-uniformly prox-regular for every $r^{\prime
}<r$,

{\rm{(b)} }the mapping $\mathrm{N}^P_{C}\cap \mathrm{B}(0, m)+\frac{m}{r}%
\operatorname*{id}$ is monotone,

{\rm{(c)}}\ there exists a maximal monotone operator $A$ defined on $H$ such
that
\[
\mathrm{N}^P_{C}(x)\cap \mathrm{B}(0, m)+\frac{m}{r}x\subset A(x)\subset\mathrm{N}^P%
_{C}(x)+\frac{m}{r}x\text{ \ for every }x\in C.
\]

\end{lmm}

\begin{proof}
The equivalence $(a)\Longleftrightarrow(b)$ is given in \cite[Theorem 4.1]%
{PRT}, while the implication $(c)\implies(b)$ is immediate. Then we only
have to prove that $(b)\implies(c).$ If $(b)$ holds, we choose a maximal
monotone operator $A$, which extends the monotone mapping $\mathrm{N}^P_{C}\cap \mathrm{B}(0,m)+\frac{m}{r}\text{id},$ such that $C\subset \text{dom}%
A\subset\overline{\text{co}}C$ (see, e.g., \cite{B}). Moreover, we have that
\begin{equation}
\mathrm{N}^P_{C}(x)\cap \mathrm{B}(0, m)+\frac{m}{r}x\subset A(x)\subset \mathrm{N}^P
_{C}(x)+\frac{m}{r}x\,\, \forall x\in C.  \label{314}
\end{equation}
Indeed, the first inclusion is obvious. If $x\in C$ and $\xi\in A(x),$ then
for\ any $y\in C$ we have $\frac{m}{r}y\in A(y)$ (since $0\in \mathrm{N}^P_{C}(y)\cap \mathrm{B}(0,m)$) and, so, $\langle\xi-\frac{m}{r}y,x-y\rangle\geq0.$
This implies
\begin{equation*}
\langle\xi-\frac{m}{r}x,y-x\rangle\leq\frac{m}{r}||y-x||^{2},
\end{equation*}
which proves that\ $\xi-\frac{m}{r}x\in\mathrm{N}^P_{C}(x)$, for
every $\xi\in A(x).$ Hence, $A(x)\subset
\mathrm{N}^P_{C}(x)+\frac{m}{r}x$, for every $x\in C$.
\end{proof}%\bigskip
\bigskip
 We also need some properties of the solution of (\ref{101}).
 The assertions of the  following lemma are very natural and may have already appeared in the literature. For the convenience of the reader, we give a proof.
\begin{lmm}
\label{prop3.2}If $x(\cdot;x_{0})$ is a solution of ${\rm (\ref{101})}$, then
for a.e. $t\in\lbrack0,T]$ we have
\begin{equation}
\langle\dot{x}(t),f(x(t))-\dot{x}(t)\rangle=0,\label{e1}%
\end{equation}%
\begin{equation}
\left\Vert f(x(t))-\dot{x}(t)\right\Vert \leq\left\Vert f(x(t))\right\Vert
,\label{e2}%
\end{equation}%
\begin{equation}
||\dot{x}(t)||\leq\min\{||f(x(t))||,\,||f(x_{0})||e^{\kappa t}\},\,\,||x(t)-x_{0}%
||\leq t||f(x_{0})||e^{\kappa t}.\label{e3}%
\end{equation}
Consequently, $x(\cdot;x_{0})$ is the unique solution of {\rm(\ref{101})} on
$[0,T].$
\end{lmm}

\begin{proof}
Let $t\in(0,T]$ be a differentiability point of the solution $x(\cdot).$
Then there is some $\delta>0$ such that
\begin{equation*}
\langle f(x(t))-\dot{x}(t),x(s)-x(t)\rangle\leq\delta||x(s)-x(t)||^{2},\text{
for all }s\in\lbrack0,T],
\end{equation*}
and, so, by dividing on $s-t$ and taking the limit as $s\downarrow t$ we
derive that
\begin{equation*}
\langle f(x(t))-\dot{x}(t),\dot{x}(t)\rangle\leq0.
\end{equation*}
Similarly, when $s\uparrow t$ we get $\langle f(x(t))-\dot{x}(t),\dot {x}%
(t)\rangle\geq0,$ which yields $(\ref{e1})$. Since $f(x(t))-\dot{x}(t)\in%
\mathrm{N}_{C}(x(t))$ and $\dot{x}(t)\in\mathrm{T}_{C}^{B}(x(t)),$ statement
$(\ref{e1})$ means that $f(x(t))-\dot{x}(t)=\Pi_{\mathrm{N}%
_{C}(x(t))}(f(x(t)))$ and this yields $(\ref{e2})$, $\left\Vert f(x(t))-%
\dot {x}(t)\right\Vert \leq\left\Vert f(x(t))\right\Vert .$ Moreover, using $%
(\ref{e1})$, we have\ (for a.e. $t\in\lbrack0,T]$)
\begin{equation}
||\dot{x}(t)||^{2}=\langle\dot{x}(t),\dot{x}(t)\rangle=\langle\dot {x}%
(t),f(x(t))\rangle\leq\left\Vert \dot{x}(t)\right\Vert \left\Vert
f(x(t))\right\Vert,  \label{de}
\end{equation}
which gives us $||\dot{x}(t)||\leq||f(x(t))||.$ Then
\begin{equation*}
\begin{split}
\frac{d}{dt}\left\Vert x(t)-x_{0}\right\Vert ^{2}=2\langle x(t)-x_{0},\dot {x%
}(t)\rangle & \leq2||x(t)-x_{0}||\, ||f(x(t))|| \\
& \leq2||x(t)-x_{0}||(||f(x_{0})||+\kappa ||x(t)-x_{0}||) \\
& =2||f(x_{0})||\,||x(t)-x_{0}||+2\kappa ||x(t)-x_{0}||^{2},
\end{split}%
\end{equation*}
which by Lemma $\ref{l2}$ gives us
\begin{equation}
||x(t)-x_{0}||\leq\frac{||f(x_{0})||}{\kappa}(e^{\kappa
t}-1)\leq||f(x_{0})||te^{\kappa t},  \label{307}
\end{equation}
so that, using the inequality of the middle together with ($\ref{de}$),
\begin{equation*}
\begin{split}
||\dot{x}(t)||& \leq||f(x(t))||\leq||f(x_{0})||+\kappa ||x(t)-x_{0}|| \\
&\leq ||f(x_{0})||+||f(x_{0})||(e^{\kappa t}-1)=||f(x_{0})||e^{\kappa t}.
\end{split}%
\end{equation*}
This proves\ $(\ref{e2})$ and $(\ref{e3})$.

To finish we need to check the uniqueness of the solution. Proceeding by
contradiction, we assume that $y(\cdot)$ is another solution on $[0,T]$ of $(%
\ref{101})$. Then for all $t\in\lbrack0,T]$ such that\ $%
||f(x(t))||+||f(y(t)||>0$ and $f(y(t))-\dot{y}(t)\in\mathrm{N}_{C}(y(t))$ we
have
\begin{equation*}
\frac{f(y(t))-\dot{y}(t)}{||f(x(t))||+||f(y(t)||}\in\mathrm{N}_{C}(y(t))\cap%
\mathbb{B},
\end{equation*}
and similarly for $x(\cdot).$ Then, by the $r$-uniformly prox-regularity
hypothesis on $C,$
\begin{equation}
\langle\dot{x}(t)-\dot{y}(t),x(t)-y(t)\rangle\leq\left( \kappa+\frac{1}{r}%
\big(||f(x(t))||+||f(y(t))||\big)\right) ||x(t)-y(t)||^{2};  \label{ok}
\end{equation}
this inequality also holds when $||f(x(t))||+||f(y(t)||=0$ as a consequence
of ($\ref{de}$). By applying Gronwall's Lemma (Lemma $\ref{l2}$) with the
function $\frac{1}{2}||x(t)-y(t)||^{2},$ and observing that $x(0)=y(0)=x_{0}$%
, it follows that $x(t)=y(t)$ for every $t\in\lbrack0,T]$.
\end{proof}

The main result is given in the following theorem, using\ a convex analysis
approach, while Theorem \ref{theo3.1-b} below provides more properties of
the solution, which will be used later on.

\begin{theorem}
\label{theo3.1} The differential inclusion {\rm(\ref{101})} has a
unique solution $x(\cdot ;x_{0})$ starting at $x_{0}\in C,$ which
is Lipschitz on every bounded interval.
\end{theorem}

\begin{proof}
We fix a sufficiently large\ $m>0$ and choose a $T_{0}>0$ such that
\begin{equation}
||f(x_{0})||+\kappa \big(||f(x_{0})||T_{0}e^{(\kappa +\frac{m}{r})T_{0}}+1%
\big)\leq m.  \label{312}
\end{equation}
By Lemma \ref{extension}(c) we consider a maximal monotone extension $A$
such that, for all $x\in C,$
\begin{equation}
\mathrm{N}_{C}(x)\cap \mathrm{B}(0, m)+\frac{m}{r}x\subset A(x)\subset\mathrm{N}%
_{C}(x)+\frac{m}{r}x.  \label{exex}
\end{equation}
According to \cite{B}, the differential inclusion
\begin{equation}
\begin{cases}
\dot{x}(t)\in f(x(t))+\frac{m}{r}x(t)-A(x(t)),\,\,\text{a.e.}\,\,\,t\in
\lbrack0,T_{0}] \\
x(0)=x_{0}\in C,%
\end{cases}
\label{315}
\end{equation}
has\ a unique solution $x(\cdot)$ such that $x(t)\in\text{dom}A$ $(\subset%
\overline{\text{co}}C)$ for all $t\in\lbrack0,T_{0}]$, as well as (see,
e.g., \cite{AHT1})
\begin{equation*}
\left\Vert \frac{d^{+}x(t)}{dt}\right\Vert \leq e^{(\kappa +\frac{m}{r}%
)t}\left\Vert \frac{d^{+}x(0)}{dt}\right\Vert \leq e^{(\kappa +\frac{m}{r}%
)t}||\Pi_{A(x_{0})}(f(x_{0})+\frac{m}{r}x_{0})||.
\end{equation*}
Moreover, since $\frac{m}{r}x_{0}\in A(x_{0})$ (due to (\ref{exex})), for
all $t\in\lbrack0,T_{0}]$
\begin{equation*}
\left\Vert \frac{d^{+}x(t)}{dt}\right\Vert \leq e^{(\kappa+\frac{m}{r}%
)t}||f(x_{0})||\leq e^{(\kappa+\frac{m}{r})T_{0}}||f(x_{0})||=:k,
\end{equation*}
and, hence,
\begin{equation}
||x(t)-x_{0}||\leq kT_{0},  \label{ee}
\end{equation}%
\begin{equation*}
||f(x(t))||\leq||f(x_{0})||+\kappa||x(t)-x_{0}||\leq||f(x_{0})||+\kappa
kT_{0};
\end{equation*}
in particular, $x(\cdot)$ is $k$-Lipschitz on $[0,T_{0}].$
Next, we want to show that there exists some $T' \in (0, T_0]$ such that $x(t)\in C$ for every $t\in\lbrack0,T'].$ For
this aim we shall apply Proposition $\ref{co.2}.$ Given $y\in C\cap
\mathrm{B}(x_{0},kT_{0}+1)$ and $\xi\in\mathrm{N}_{C}(y)$, we define $z:=\Pi _{%
\mathrm{N}_{C}(y)}(f(y))\in\mathrm{N}_{C}(y)$ ($z$ is well defined since $%
\mathrm{N}_{C}(y)$ is closed (and convex)). It is easy to see that
\begin{equation*}
||z||\leq||f(y)||\leq||f(x_{0})||+\kappa||y-x_{0}||\leq||f(x_{0})||+\kappa
(kT_{0}+1)\leq m.
\end{equation*}
Hence, according to $(\ref{exex})$, we derive that $y^{\ast}:=z+\frac{m}{r}%
y\in\mathrm{N}_{C}(y)\cap \mathrm{B}(0,m)+\frac{m}{r}y\subset A(y),$ with $%
\left\Vert y^{\ast}\right\Vert \leq\overline{m}:=m(1+\frac{1}{r}(\left\Vert
x_{0}\right\Vert +kT_{0}+1)).$ \\
Now, since $f(y)-z\in\mathrm{T}_{C}(y)$ we obtain that $\langle\xi
,f(y)-z\rangle\leq0,$ which shows that
\begin{equation}
\inf_{v^{\ast}\in A(y)\cap \mathrm{B}(0,\overline{m})}\left\langle \xi,f(y)+\frac{m}{r%
}y-v^{\ast}\right\rangle \leq\langle\xi,f(y)+\frac{m}{r}y-y^{\ast}\rangle%
\leq0.  \label{317}
\end{equation}
Consequently, according to Proposition $\ref{co.2}$, there is a positive
number $T^{\prime}\in(0,T_{0})$ such that $x(t)\in C$ for every $t\in\lbrack
0,T^{\prime}].$ For every $t \in [0, T']$, $(\ref{exex})$ implies that
\begin{equation*}
\begin{split}
\dot{x}(t)\in f(x(t))+\frac{m}{r}x(t)-A(x(t))& \subset f(x(t))+\frac{m}{r}%
x(t)-\mathrm{N}_{C}(x(t))-\frac{m}{r}x(t) \\
&=f(x(t))-\mathrm{N}_{C}(x(t));
\end{split}%
\end{equation*}
that is, $t \mapsto x(t), t \in [0, T']$ is a solution of\ $(\ref{101})$ on $[0,T'].$
Now, we set
\begin{equation*}
T:=\sup\left\{ T^{\prime}>0\,\,\text{such that system $(\ref{101})$ has a
solution }x(\cdot;x_{0})\text{ on $[0,T^{\prime}]$}\right\} ;
\end{equation*}
so, $T>0$ from the paragraph above. If $T$ is finite, then we take a
sequence $(T_{n})$ such that $T_{n}\uparrow T,$ and denote $%
x_{n}(\cdot;x_{0})$ the corresponding solution of $(\ref{101}),$ which is
defined on $[0,T_{n}].$ Let\ function $x(\cdot;x_{0}):[0,T)\rightarrow H$ be
defined as
\begin{equation*}
x(t;x_{0})=x_{n}(t)\text{ \ if }t\leq T_{n}.
\end{equation*}
According to Lemma \ref{prop3.2} (relation (\ref{e3})), this function is a
well-defined Lipschitz continuous function on $\left[ 0,T\right) ,$ with
Lipschitz constant equal to $\,||f(x_{0})||e^{\kappa T}.$ Thus, we can
extend continuously function $x(\cdot;x_{0})$ to $\left[ 0,T\right] $ by
setting $x(T):=\underset{n\rightarrow\infty}{\lim}x(T_{n}).$ Since $x(T)\in
C,$ from the first paragraph we find a $T_{1}>0$ and a solution of $(\ref%
{101})$ on $[0,T+T_{1}]$ which coincides with $x(\cdot;x_{0})$ on $[0,T],$
contradicting the finiteness of $T$--this is to say that $T=\infty.$
\end{proof}

\bigskip

An immediate consequence of (the proof of) Theorem $\ref{theo3.1}$ is that
the solution of differential inclusion ($\ref{101})$ satisfies the so-called
semi-group property,%
\begin{equation}
x(t;x(s;x_{0}))=x(t+s;x_{0})\text{ for all }t,s\geq0\text{ and }x_{0}\in C.
\label{sgp}
\end{equation}
The following theorem gathers further properties of the solution of $(\ref%
{101}),$ that we shall use in the sequel. Relation $(\ref{308})$ below on
the derivative\ of the solution reinforces the statement of\ Lemma $\ref%
{prop3.2}.$ \\
The following properties are well-known when the set $C$ is
convex, since in this case $A=\mathrm{N}_{C}$ is a maximal
monotone operator (see \cite{B}).
%and we extend it to Prox-regular
%sets??} \textcolor{red}{yes, I think it is nice}

%Some of these properties (the first part in (a) and also (b))
%have been established in \cite{L}, but in our case, they are simple
%consequences thanks to the representation of $(\ref{101})$\ as a differential
%inclusion with a maximal monotone operator.

\begin{theorem}
\label{theo3.1-b} Let $x(\cdot;x_{0}),$ $x_{0}\in C,$ be the solution of
{\rm(\ref{101})}.\ Then the following statements hold true{\rm:}

{\rm(a) }For every $t\geq0,$ $x(\cdot;x_{0})$ is right-derivable at $t$
with\
\begin{equation}
\begin{split}
\frac{d^{+}x(t)}{dt}&=[f(x(t))-\mathrm{N}_{C}(x(t))]^{\circ} \\
&=f(x(t))-\Pi _{\mathrm{N}_{C}(x(t))}(f(x(t))=\Pi_{\mathrm{T}%
_{C}(x(t))}(f(x(t))),
\end{split}
\label{308}
\end{equation}%
\begin{align}
& \left\Vert \frac{d^{+}x(t)}{dt}\right\Vert\leq\min\{||f(x(t))||,||f(x_{0})||e^{\kappa t}\},
\label{311} \\
& \left\Vert \frac{d^{+}x(t)}{dt}\right\Vert \leq\left\Vert\frac{d^{+}x(0)}{dt}\right\Vert e^{\kappa t+\frac {%
2||f(x_{0})||}{\kappa r}(e^{\kappa t}-1)}.  \label{311-2}
\end{align}

{\rm(b) }The mapping $t\rightarrow\frac{d^{+}x(t)}{dt}$ is
right-continuous on $[0,T).$

{\rm(c) }If $y(\cdot;y_{0}),$ $y_{0}\in C,$ is the corresponding solution
of {\rm(\ref{101})}, then for every $t\geq0$%
\begin{equation*}
\left\Vert x(t)-y(t)\right\Vert \leq\left\Vert x_{0}-y_{0}\right\Vert e^{\kappa t+%
\frac{||f(x_{0})||+||f(y_{0})||}{\kappa r}(e^{\kappa t}-1)}.
\end{equation*}
\end{theorem}

\begin{proof}
We fix $t\geq0$ (we may suppose that $t=0$)$.$ From the argument used in the
proof of Theorem $\ref{theo3.1}$ we know that for some $m>||f(x_{0})||+%
\kappa $ ($\kappa $ is the Lipschitz constant of $f$) there exists a maximal
monotone operator $A$ such that $x(\cdot):=x(\cdot;x_{0})$ is the solution
of the following differential inclusion on some interval $\left[ 0,\delta%
\right] ,$ $\delta>0, $
\begin{equation*}
\dot{x}(t)\in f(x(t))+\frac{m}{r}x(t)-A(x(t)),\text{ }x(0)=x_{0},
\end{equation*}
where $r$ comes from the $r$-uniform prox-regularity of $C.$ W.l.o.g. we may
suppose that $||f(x(t))||+\kappa <m$ for all $t\in$ $\left[ 0,\delta\right] $
so that (see Proposition $\ref{Brezis}$), for every $t\in\left[ 0,\delta%
\right] ,$%
\begin{equation}
\frac{d^{+}x(t)}{dt}=[f(x(t))+\frac{m}{r}x(t)-A(x(t))]^{\circ}.  \label{tt}
\end{equation}
Since $f(x(t))\in \mathrm{B}(0,m)$, we have that
\begin{align*}
\left( f(x(t))-\mathrm{N}_{C}(x(t))\right) ^{\circ} & =f(x(t))-\Pi _{\mathrm{N}%
_{C}(x(t))}(f(x(t))) \\
& =f(x(t))-\Pi_{\mathrm{N}_{C}(x(t))\cap \mathrm{B}(0,m)}(f(x(t))) \\
& =\left( f(x(t))-\mathrm{N}_{C}(x(t))\cap \mathrm{B}(0,m)\right) ^{\circ},
\end{align*}
and, so, due to $(\ref{tt})$, and the inclusions ($\ref{exex}$):
\begin{equation*}
f(x(t))-\mathrm{N}_{C}(x(t))\cap \mathrm{B}(0,m)\subset f(x(t))+\frac{m}{r}%
x(t)-A(x(t))\subset f(x(t))-\mathrm{N}_{C}(x(t)),
\end{equation*}
we get the first equality in $(\ref{308}).$ The other two equalities in $(%
\ref{308})$ easily follow from the definition of the orthogonal projection.
Moreover, statement (b) is also a consequence of Proposition $\ref{Brezis}.$
Thus, $(\ref{308})$ follows from Lemma $\ref{prop3.2}.$ Finally, $(\ref%
{311-2})$ and statement (c) follow easily using relation $(\ref{ok})$ (and
Lemma $\ref{l2}$).
\end{proof}

\smallskip

The main idea behind the previous existence theorems, Theorems $\ref{theo3.1}
$ and $\ref{theo3.1-b}$, as well as the forthcoming results on Lyapunov
stability in the next section, is that differential inclusion (\ref{101}) is
in some sense equivalent to a differential inclusion governed by a
(Lipschitz continuous perturbation of a) maximal monotone operator. This
fact is highlighted in the following corollary. Recall, by Lemma \ref%
{extension}(c), that for every $m>0$ the $r$-uniformly prox-regularity of
the set $C$ yields the existence of a maximal monotone operator $A_C$ such
that
\begin{equation}  \label{22}
\mathrm{N}_{C}(x)\cap \mathrm{B}(0,m)+\frac{m}{r}x\subset A_C(x)\subset\mathrm{N}%
_{C}(x)+\frac{m}{r}x\text{ \ for every }x\in C.
\end{equation}

\begin{corollary}
An absolutely continuous function $x(t)$ is a solution of $(\ref{101})$ on $[0, T]$%
; that is,
\begin{equation}
\begin{cases}
\dot{x}(t) \in f(x(t)) - \mathrm{N}_C(x(t)), \text{a.e.}\, \, t \in [0, T] \\
x(0) =x_0 \in C,%
\end{cases}
\notag
\end{equation}
if and only if it is (the unique) solution of the following differential
inclusion, for some $m>0$,
\begin{equation}
(DIM)\quad
\begin{cases}
\dot{x}(t) \in f(x(t))+\frac{m}{r}x(t) - A_C(x(t)), \text{a.e.}\, \, t \in [0,
T] \\
x(0) =x_0 \in C,%
\end{cases}
\notag
\end{equation}
where the maximal monotone operator $A_C: H\rightrightarrows H $ is defined
in $(\ref{22})$.
\end{corollary}

\begin{proof}
According to Theorems $\ref{theo3.1}$ and $\ref{theo3.1-b}$ (namely, $(\ref%
{311})$), differential inclusion $(101)$ has a unique (absolutely
continuous) solution $x(t):=x(t;x_0)$ which satisfies $\left\Vert\frac{d^{+}x(t)}{dt}%
\right\Vert\leq ||f(x_{0})||e^{\kappa T} \text{ for a.e. } t\in [0, T]$. Then, we
find an $m>0 $ such that
\begin{equation*}
\dot{x}(t) \in f(x(t))- \mathrm{N}_C(x(t))\cap \mathrm{B}(0, m),
\end{equation*}
and, so, by the definition of $A_C$ above (see $(\ref{22})$) we conclude
that $x(t)$ is also the solution of differential inclusion (DIM).

Conversely, if $x(t)$ is a solution of differential inclusion (DIM) for some
$m>0$, then, as it follows from the proof of Theorem $\ref{theo3.1}$, we get
that $x(t)\in C$ for all $t\in [0,T_0]$ for some $T_0>0$. Hence, once again
by $(\ref{22})$, we conclude that $x(t)$ is also a solution of $(101)$ on $%
[0,T_0]$. Taking into account Lemma $\ref{theo3.1}$ we show, also as in the
proof of Theorem $\ref{theo3.1}$, that $T_0$ can be taken to be $T$.
\end{proof}
\section{Lyapunov stability analysis}

In this section, we give explicit characterizations for lsc $a$-Lyapunov
pairs, Lyapunov functions, and invariant sets\ associated to\ differential
inclusion $(\ref{101}).$ Recall that $x(\cdot;x_{0})$ (or $x(\cdot),$ when
any confusion is excluded) refers to the unique solution of $(\ref{101}),$
which satisfies\ $x(0;x_{0})=x_{0}.$

\begin{dfntn}
\label{defini41} Let functions $V,W:H\rightarrow\overline{\mathbb{R}}$ be lsc,
with $W\geq0$, and let an $a\geq0.$ We say that $(V,W)$ is (or forms) an
$a$-Lyapunov pair for differential inclusion {\rm (\ref{101})} if,\ for all
$x_{0}\in C,$
\begin{equation}
e^{at}V(x(t;x_{0}))+\int_{0}^{t}W(x(\tau;x_{0}))d\tau\leq V(x_{0}%
)\,\,\,\text{for all}\,\,t\geq0.\label{401}%
\end{equation}
In particular, if $a=0,$ we say that $(V,W)$ is a Lyapunov pair. If,\ in
addition, $W=0,$ then $V$ is said to be a Lyapunov function.

A closed set $S\subset C$ is said to be invariant for {\rm(\ref{101})} if
the function $\delta_{S}$ is a Lyapunov function.
\end{dfntn}

Equivalently, using (\ref{sgp}), it is not difficult to show that $a$%
-Lyapunov pairs are those pairs of functions $V,W:H\rightarrow \overline{%
\mathbb{R}}$ such that the mapping\ $t\rightarrow
e^{at}V(x(t;x_{0}))+\int_{0}^{t}W(x(\tau ,x_{0}))d\tau $ is nonincreasing.
In other words (see, e.g. \cite[Proposition 3.2]{AHT1}), for any $x_{0}\in C$%
, there exists $t>0$ such that
\begin{equation}
e^{as}V(x(s;x_{0}))+\int_{0}^{s}W(x(\tau; x_{0}))d\tau \leq V(x_{0})\,\,\,%
\text{for all}\,\,s\in \lbrack 0,t].  \label{equi}
\end{equation}

The failure of \textit{regularity} in our Lyapunov candidate-like pairs is
mainly carried out by the function $V$, since the function $W$ can be always
regularized to a Lipschitz continuous function on every bounded subset of $H$
as the following lemma shows (see, e.g., \cite{C}).

\begin{lmm}
\label{lemm4.1} Let $V$,\ $W$ and $a$ be as in Definition {\rm\ref{defini41}%
}. Then there exists a sequence of lsc functions $W_{k}:H\rightarrow
\mathbb{R}$, $k\geq1,$ converging pointwisely to $W$ (for instance,
$W_{k}\nearrow W$) such that $W_{k}$ is Lipschitz continuous on every bounded
subset of $H.$ Consequently, $(V,W)$ forms an $a$-Lyapunov pair for
{\rm(\ref{101})} if and only if each $(V,W_{k})$ does.
\end{lmm}

Now, we give the main theorem of this section, which characterizes lsc $a$%
-Lyapunov pairs associated to differential inclusion\ $(\ref{101}).$

\begin{theorem}
\label{theor4.1}Let functions $V,W:H\rightarrow \overline{\mathbb{R}}$ be
lsc, with $W\geq 0$ and $\normalfont{\text{dom}}V\subset C,$ $a\geq 0,$ and let $%
x_{0}\in \normalfont{\text{dom}}V.$ If there is $\rho >0$ such that, for any\ $x\in
\mathrm{B}(x_{0},\rho ),$
\begin{equation}
\sup_{\xi \in \partial _{P}V(x)}\,\,\underset{x^{\ast }\in \mathrm{N}_{C}(x)\cap
\mathrm{B}(0 ,||f(x)||)}{\min }\langle \xi ,f(x)-x^{\ast }\rangle
+aV(x)+W(x)\leq 0,  \label{locally lyapunov}
\end{equation}%
then\ there is some\ $T^{\ast }>0$ such that
\begin{equation*}
e^{at}V(x(t;x_{0}))+\int_{0}^{t}W(x(\tau ;x_{0}))d\tau \leq V(x_{0})\,\, \forall
\,\,t\in \lbrack 0,T^{\ast }].
\end{equation*}%
Consequently, the following statements are equivalent provided that either $%
\partial \equiv \partial _{P}$ or $\partial \equiv \partial _{F}:$

$(i)$ $(V,W)$ is an $a$-Lyapunov pair for {\rm(\ref{101})}$;$

$(ii)$ for every $x\in \normalfont{\text{dom}}V$ and $\xi \in \partial V(x);$
\begin{equation*}
\langle \xi ,(f(x)-\mathrm{N}_{C}(x))^{\circ}\rangle +aV(x)+W(x)\leq 0;
\end{equation*}

$(iii)$ for every $x\in \normalfont{\text{dom}}V$ and $\xi \in \partial V(x);$
\begin{equation*}
\underset{x^{\ast }\in \mathrm{N}_{C}(x)\cap \mathrm{B}(0 ,||f(x)||)}{\min }%
\langle \xi ,f(x)-x^{\ast }\rangle +aV(x)+W(x)\leq 0;
\end{equation*}

$(iv)$ for every $x\in \normalfont{\text{dom}}V;$
\begin{equation*}
V^{\prime }(x;(f(x)-\mathrm{N}_{C}(x))^{\circ})+aV(x)+W(x)\leq 0;
\end{equation*}

$(v)$ for every $x\in \normalfont{\text{dom}}V;$
\begin{equation*}
\underset{x^{\ast }\in \mathrm{N}_{C}(x)\cap \mathrm{B}(0 ,||f(x)||)}{\inf }%
V^{\prime }(x;f(x)-x^{\ast })+aV(x)+W(x)\leq 0.
\end{equation*}%
Moreover, when\ $H$ is finite-dimensional, all the statements above except $%
(ii)$ are equivalent when $\partial =\partial _{L}.$
\end{theorem}

\begin{proof}
Let us start with the first part of the theorem. We choose $T>0$ such that
\begin{equation*}
T||f(x_{0})||e^{\kappa T}\leq \frac{\rho }{2},
\end{equation*}%
and put
\begin{align*}
& k:=2\max \{||f(x_{0})||e^{\kappa T},||f(x_{0})||+\kappa Te^{\kappa
T}||f(x_{0})||+\kappa+1\}; \\
& m:=k+\frac{k}{r}(||x_{0}||+\rho ).
\end{align*}%
Thanks to\ Lemma $\ref{lemm4.1}$ we shall assume in what follows that $W$ is
Lipschitz continuous on $\mathrm{B}(x_{0},\rho ).$ As before we denote $x(\cdot )$
the solution of $(\ref{101})$ on $[0,T]$ satisfying\ $x(0)=x_{0}.$ According
to Theorem\ $\ref{theo3.1-b},$ for a.e. $t\in \lbrack 0,T]$ we have $||\dot{x%
}(t)||\leq ||f(x(t))||$ and, due to the $\kappa$-Lipschitzianity of $f,$
\begin{align*}
2||f(x(t))||& \leq 2||f(x_{0})||+2\kappa ||x(t)-x_{0}|| \\
& <2\max \{||f(x_{0})||e^{\kappa T},||f(x_{0})||+\kappa Te^{\kappa
T}||f(x_{0})||+\kappa +1\}=k;
\end{align*}%
that is, $\dot{x}(t)\in f(x(t))-\left( \mathrm{N}_{C}(x(t))\cap \mathrm{B}(0
,k)\right).$ Hence, if $A:H\rightrightarrows H$ is the monotone\ operator
defined as
\begin{equation*}
A(x):=%
\begin{cases}
\mathrm{N}_{C}(x)\cap \mathrm{B}(0,k)+\frac{k}{r}x & \text{if}\,\,x\in C, \\
\emptyset & \text{otherwise,}%
\end{cases}%
\end{equation*}%
and $\hat{A}$ is one maximal monotone extension of it, then it is
immediately seen that $x(\cdot )$ is also the unique solution of
the following differential inclusion,
\begin{equation*}
\begin{cases}
\dot{x}(t)\in f(x(t))+\frac{k}{r}x(t)-\hat{A}(x(t)),\,\,\text{a.e.}\,\,\,t\in
\lbrack 0,T], &  \\
x(0)=x_{0}, &
\end{cases}%
\end{equation*}%
where $r>0$ is the constant of the $r$-uniform prox-regularity of
the set $C.$
 We now consider the differential inclusion
\begin{equation*}
\begin{cases}
\dot{x}(t)\in f(x(t))+\frac{k}{r}x(t)-\hat{A}(x(t))\,\, t \ge 0, &  \\
x(0)= x_0 &
\end{cases}%
\end{equation*}
(recall that $x_0\in C \subset \text{dom}A \subset
\text{dom}\hat{A}$).

We prove the existence of some $\bar{t}\in (0,T]$ such that
$$ e^{at}V(x(t;x_0))+ \int^t_0W(x(\tau; x_0))d\tau \le V(x_0) \,\,\forall t \in [0, \bar{t}].$$
Since $\Vert \hat{A}^{\circ}(x) \Vert \leq
d(0,\mathrm{N}_{C}(x)\cap \mathrm{B}(0 ,k)+\frac{k}{r}x) \leq
\frac{k}{r}\Vert x\Vert$, for all $x\in C$, according to Lemma
\ref{extension}, there exist $M >0, \rho' \in (0, \rho)$ such that
for all $ x \in \mathrm{B}(x_0, \rho') \cap \text{dom}V$ $\subset
C $, one has
\begin{equation}\label{4.01}
\max\Big\{\Vert \hat{A}^{\circ}(x) \Vert, \Vert
\Pi_{\hat{A}(x)}(f(x)+ \frac{k}{r}x) \Vert, \Vert f(x) \Vert +
\frac{k}{r}\Vert x \Vert  \Big\} \le \,\,M\,\,\text{and}\,\, \Vert
f(x) \Vert \le k;
\end{equation}
that is, $(V+\mathrm{I}_{\hat{A}_M})(x)= V(x),\,\, \forall x \in
\mathrm{B}(x_0, \rho') \cap \text{dom}V,$ where $\hat{A}_M:= \{x
\in \text{dom}\hat{A}\mid \,\, \Vert \hat{A}^{\circ}(x) \Vert \le
M \},$ and so
$$\partial_P(V+\mathrm{I}_{\hat{A}_M})(x)= \partial_PV(x),\,\,\forall x \in \mathrm{B}(x_0, \frac{\rho'}{2}) \cap \text{dom}V.$$
We now fix $x \in \mathrm{B}(x_0, \frac{\rho'}{2}) \cap
\text{dom}V$ and $\xi \in
\partial_P(V+\mathrm{I}_{A_M})(x)=\partial_PV(x)$. From
(\ref{locally lyapunov})  there exists $x^* \in \mathrm{N}_C(x)
\cap \mathrm{B}(0, \Vert f(x)\Vert)$ such that
\begin{equation*}
\langle \xi, f(x)- x^* \rangle + aV(x) + W(x) \le 0.
\end{equation*}
Since
\begin{equation*}
\begin{split}
x^* +\frac{k}{r}x & \in \mathrm{N}_C(x) \cap \mathrm{B}(0, \Vert f(x)\Vert) + \frac{k}{r}x \subset \mathrm{N}_C(x) \cap \mathrm{B}(0, k) + \frac{k}{r}x  \subset \hat{A}(x),
\end{split}
\end{equation*}
and $ \Vert x^* +\frac{k}{r}x \Vert \le \Vert f(x)\Vert +
\frac{k}{r}\Vert x \Vert \le M $ (recall (\ref{4.01})), we have
that
$$x^* +\frac{k}{r}x \in \hat{A}(x) \cap \mathrm{B}(0, M). $$
In other words, for all $x \in \mathrm{B}(x_0, \frac{\rho'}{2})
\cap \text{dom}V$ and $\xi \in
\partial_P(V+\mathrm{I}_{A_M})$ it holds that
\begin{equation}\label{4.02}
\begin{split}
& \underset{z^* \in \hat{A}(x) \cap \mathrm{B}(0, M)}{\inf}\langle \xi, f(x)+ \frac{k}{r}x - z^* \rangle +aV(x) +W(x) \\
& \, \quad \,\quad \, \quad  \le \langle \xi, f(x)+ \frac{k}{r}x - (x^* + \frac{k}{r}x) \rangle  +aV(x) +W(x) \\
& \,\quad \,\quad \, \quad  = \langle \xi, f(x)-x^* \rangle
+aV(x) +W(x) \le 0.
\end{split}
\end{equation}
Hence, using a similar argument as in the proof of \cite[Theorem
9]{AHB}, condition (\ref{4.02}) ensures the existence of some
$\bar{t} \in (0, T]$ such that
$$e^{at}V(x(t;x_0))+ \int^t_0W(x(\tau; x_0))d\tau \le V(x_0) \,\,\forall t \in [0, \bar{t}],$$
which proves the first part of the theorem.
\newline

%\vskip{0.01cm}
 We turn now to the second part of the
theorem. Implications $(iv)\Rightarrow
(v)$ and $(ii)\Rightarrow (iii)$ follow from the relation $(f(x)-\mathrm{N}%
_{C}(x))^{\circ}=f(x)-\Pi _{\mathrm{N}_{C}(x)}(f(x)),$ $x\in C,$ and the fact
that $\left\Vert \Pi _{\mathrm{N}_{C}(x)}(f(x))\right\Vert \leq \left\Vert
f(x)\right\Vert .$

$(i)\Rightarrow(iv).$ Assuming that\ $(V,W)$ is an $a$-Lyapunov pair, for
any $x_{0}\in\text{dom}V$ and $t>0$ the solution $x(\cdot)=x(\cdot ;x_{0}) $
satisfies%
\begin{equation}
0\geq
t^{-1}(V(x(t))-V(x_{0}))+t^{-1}(e^{at}-1)V(x(t))+\int_{0}^{t}t^{-1}W(x(%
\tau))d\tau.  \label{404}
\end{equation}
Thus, observing that $\frac{x(t)-x_{0}}{t}\rightarrow\lbrack f(x_{0})-%
\mathrm{N}_{C}(x_{0})]^{\circ}$ (recall Theorem $\ref{theo3.1-b}(a)$), and using
the lsc of $V$ and $W,$ as $t\downarrow0$ in the last inequality we get
\begin{align*}
V^{\prime}(x_{0},(f(x_{0})-\mathrm{N}_{C}(x_{0}))^{\circ}) & =\underset{\underset%
{t\downarrow0}{w\rightarrow(f(x_{0})-\mathrm{N}_{C}(x_{0}))^{\circ}}}{\lim\inf}%
\frac{V(x_{0}+tw)-V(x_{0})}{t} \\
& \leq\underset{t\downarrow0}{\lim\inf}t^{-1}(V(x(t))-V(x_{0}))\leq
-aV(x_{0})-W(x_{0}).
\end{align*}

$(iv)\Rightarrow(ii)$ and $(v)\Rightarrow(iii),$ when $\partial=\partial_{F}$
or $\partial=\partial_{P}.$ These implications follow due to the relation $%
\langle\xi,v\rangle\leq V^{\prime}(x,v)$ for all $\xi\in\partial_{F}V(x)$, $%
x\in\text{dom}V,$ and\ $v\in H$.

$(iii)\Rightarrow (i)$ is an immediate consequence of the first part of the
theorem together with $(\ref{equi}).$

Finally, to prove the last statement of the theorem when \ $\partial
=\partial_{L}$ in the finite-dimensional case, we first check that $%
(i)\implies(iii).$ Assume that $(i)$ holds and take $x\in\text{dom}V$
together with $\xi\in\partial_{L}V(x)$, and let sequences $x_{k}\overset{V}{%
\rightarrow}x$ together with $\xi_{k}\in\partial_{P}V(x_{k})$ such that $%
\xi_{k}\rightarrow\xi.$ Since $(iii)$ holds for $\partial=\partial_{P},$ for
each $k$ there exists $x_{k}^{\ast}\in\mathrm{N}_{C}(x_{k})\cap
\mathrm{B}(0,||f(x_{k})||)$ such that
\begin{equation*}
\langle\xi_{k},f(x_{k})-x_{k}^{\ast}\rangle+aV(x_{k})+W(x_{k})\leq0.
\end{equation*}
We may assume that $(x_{k}^{\ast})$ converges to some $x^{\ast}\in \mathrm{N}%
_{C}(x)\cap \mathrm{B}(0,||f(x)||)$ (thanks to the $r$-uniform prox-regularity
of $C$), which then satisfies $\langle\xi,f(x)-x^{\ast}\rangle+aV(x)+W(x)%
\leq0$ (using the lsc of the involved functions), showing that $(iii)$
holds. Thus, since $(iii)\,\, (\text{with $\partial =\partial_{P}$}%
)\implies(i)$, we deduce that $(i)\Longleftrightarrow(iii).$ This suffices
to get the conclusion of the theorem.
\end{proof}

%\begin{remark}\tcb{We have to rewrite this remark!!} \textcolor{red}{ In this theorem, we give a Lyapunov characterization for differential inclusion (\ref{101}), this characterization likes as characterizations of Theorem 9 and Corollary 10 in \cite{AHB}, however, we can not apply directly from the results of \cite{AHB}. Moreover, by properties of normal cone, we get the characterization more strictly as require $\mathrm{N}_C(x) \cap \mathrm{B}(0, \Vert f(x) \Vert) \ne \emptyset.$}
%\end{remark}

Because the solution $x(\cdot)$ of differential inclusion ($\ref{101}$)
naturally lives in $C$, it is immediate that a (lsc) function $%
V:H\rightarrow \overline{\mathbb{R}}$ is Lyapunov for ($\ref{101}$) iff the
function $V+\mathrm{I}_{C}$ is Lyapunov. Hence, Theorem $\ref{theor4.1}$
also provides the characterization of Lyapunov functions without any
restriction on their domains; for instance, accordingly to Theorem $\ref%
{theor4.1}(iii)$, $V$ is Lyapunov for ($\ref{101}$) iff for every $x\in\text{%
dom}V\cap C$ and $\xi\in\partial(V+\mathrm{I}_{C})(x)$ it holds
\begin{equation*}
\underset{x^{\ast}\in\mathrm{N}_{C}(x)\cap \mathrm{B}(0,||f(x)||)}{\min}\langle
\xi,f(x)-x^{\ast}\rangle+aV(x)+W(x)\leq0.
\end{equation*}
The point here is that this condition is not completely written by means
exclusively of the subdifferential of $V.$ Nevertheless, this condition
becomes more explicit in each time one can decompose the subdifferential set
$\partial(V+\mathrm{I}_{C})(x).$ For instance, this is the case, if $V$ is
locally Lipschitz and lower regular (particularly convex, see \cite[%
Definition 1.91]{M}). This fact is considered in Corollary $\ref{corol41}$
below. However, the following example shows that we can not get rid of the
condition $\text{dom}V\subset C,$ in general.

\begin{remark}
We consider the differential inclusion \textrm{(\ref{101})} in $\mathbb{R}%
^{2}$, with $C:=\mathbb{B}$ and $f(x,y)=(-y,x),$ whose unique solution such
that $x(0)=(1,0)$ is $x(t)=(\cos{t},\sin{t}).$ We take\ $V=\mathrm{I}_{S},$\
where
\begin{equation*}
S:=\{(1,y):y\in\lbrack0,1]\},
\end{equation*}
so that $\text{dom}V\cap C=\{(1,0)\}.$ For\ $\overline{x}:=(1,0)$ and $%
\xi:=(x,y)\in\partial_{P}V(\overline{x})=\{(x,y)|\,\,y\leq0\}\ $we have that
\begin{equation*}
\underset{x^{\ast}\in\mathrm{N}_{C}(\overline{x})\cap \mathrm{B}(0,||f(%
\overline {x}||)}{\min}\langle\xi,f(\overline{x})-x^{\ast}\rangle\leq\langle
\xi,f(\overline{x})\rangle=\langle(x,y),(0,1)\rangle=y\leq0,
\end{equation*}
which shows that condition $(iii)$ of Theorem \textrm{\ref{theor4.1}} holds.
However, it is clear that $V$ is not a\ Lyapunov function of \textrm{(\ref%
{101})}.
\end{remark}

\begin{corollary}
\label{corol41} Let $V$, $W$ and $a$ be as in Theorem  {\rm\ref{theor4.1}}$%
. $\ Then the following assertions hold{\rm:}

 {\rm(i)} If $V$ is Fr\'echet differentiable on $\normalfont{\text{dom}}V\cap C,$
then $(V,W)$ is an $a$-Lyapunov pair of differential inclusion  {\rm(\ref%
{101})} iff for every $x\in\normalfont{\text{dom}}V\cap C$
\begin{equation*}
\langle\nabla V(x),(f(x)-\mathrm{N}_{C}(x))^{\circ}\rangle+aV(x)+W(x)\leq0.
\end{equation*}

 {\rm(ii) }If $V$ is locally Lipschitz on $\normalfont{\text{dom}}V\cap C,$ then $%
(V,W)$ is an $a$-Lyapunov pair for differential inclusion  {\rm(\ref{101})}
if for every $x\in\normalfont{\text{dom}}V\cap C$
\begin{equation*}
\langle\xi,(f(x)-\mathrm{N}_{C}(x))^{\circ}\rangle+aV(x)+W(x)\leq0\,\,\forall
\xi\in\partial_{L}V(x).
\end{equation*}

 {\rm(iii) }If $H$ is of finite dimension and $V$ is regular and locally
Lipschitz on $\normalfont{\text{dom}}V\cap C,$ then $(V,W)$ is an $a$-Lyapunov pair
for differential inclusion  {\rm(\ref{101})} iff for every $x\in \normalfont{\text{%
dom}}V\cap C,$
\begin{equation*}
\langle\xi,(f(x)-\mathrm{N}_{C}(x))^{\circ}\rangle+aV(x)+W(x)\leq0\,\,\forall \xi\in\partial_{L}V(x).
\end{equation*}
\end{corollary}

\begin{proof}
$(i).$ Since $x(t)\in C$ for every $t\geq0$, we have that $(V,W)$ forms an $%
a $-Lyapunov pair for\ $(\ref{101})$ iff the pair $(V+\mathrm{I}_{C},W)$
does. Thus, since $\partial_{F}(V+\mathrm{I}_{C})(x)=\nabla V(x)+\mathrm{N}%
_{C}(x)$ for every $x\in\text{dom}V\cap C,$ according to Proposition 1.107
in \cite{M}, Theorem $\ref{theor4.1}$ ensures that $(V,W)$ is an $a$%
-Lyapunov pair of ($\ref{101}$) iff for every $x\in\text{dom}V\cap C$ and $%
\xi \in\mathrm{N}_{C}(x)$
\begin{equation}
\langle\nabla V(x)+\xi,(f(x)-\mathrm{N}_{C}(x))^{\circ}\rangle+aV(x)+W(x)\leq 0.
\label{416}
\end{equation}
Because $0\in\mathrm{N}_{C}(x)$ and $(f(x)-\mathrm{N}_{C}(x))^{\circ}\in
T_{C}^{B}(x)=(\mathrm{N}_{C}(x))^{\ast}$, it follow that this last
inequality is equivalent to $\langle\nabla V(x),(f(x)-\mathrm{N}%
_{C}(x))^{\circ}\rangle+aV(x)+W(x)\leq0.$

$(ii).$ Under the current assumption, for every $x\in V\cap C$ we have that $%
\partial_{L}(V+\mathrm{I}_{C})(x)\subset\partial_{L}V(x)+\mathrm{N}_{C}(x)$,
and we argue as in the proof of statement $(i).$

$(iii).$ In this case, we argue as above but using the relation $\partial
_{L}(V+\mathrm{I}_{C})(x)=\partial_{L}V(x)+\mathrm{N}_{C}(x).$
\end{proof}

It the result below, Theorem $\ref{theor4.1}$ is rewritten in order to
characterize invariant sets associated to differential inclusion $(\ref{101}%
) $ (see Definition $\ref{invariant}$). Criterion (iii) below is of the same nature as
the one used in \cite{CP}.

\begin{theorem}
\label{t51}Given a closed set $S\subset C$ we\ denote\ by $N_{S}$ either $%
\mathrm{N}_{S}^{P}$ or $\mathrm{N}_{S}^{F},$ and by $T_{S}$\ either $%
T_{S}^{B},$ $T_{S}^{\text{w}},$ $\overline{\text{co}}T_{S}^{\text{w}},$
or $(N_{S})^{\ast}.$ Then $S$ is an invariant set for {\rm(\ref{101}) }iff
one of the following equivalent statements hold{\rm:}

 {\rm(i) }$(f(x)-\mathrm{N}_{C}(x))^{\circ}\in T_{S}(x)\, \,\forall x\in S;$

 {\rm(ii) }$[f(x)-\mathrm{N}_{C}(x)]\cap T_{S}(x)\cap
\mathrm{B}(0,||f(x)||)\neq \emptyset\ \forall x\in S;$

 {\rm(iii) }$\underset{x^{\ast}\in\lbrack f(x)-\mathrm{N}_{C}(x)]\cap
\mathrm{B}(0,||f(x)||)}{\inf}\langle\xi,x^{\ast}\rangle\leq0\,\,\forall x\in
S,\,\,\forall\,\xi\in N_{S}(x).$
\end{theorem}

\begin{proof}
Under the invariance of $S$ we write (recall Theorem $\ref{theo3.1-b}$)
\begin{equation*}
(f(x)-\mathrm{N}_{C}(x)))^{\circ}=\frac{d^{+}x(0;x)}{dt}=\underset{t\searrow 0}{%
\lim}\frac{x(t)-x}{t}\in T_{S}^{B}(x),
\end{equation*}
showing that (i) with $T_{S}(x)=T_{S}^{B}(x)$ holds. The rest of the
implications follows by applying Theorem $\ref{theor4.1}$ with the use of
the following equalities
\begin{equation*}
T_{S}^{B}(x)\subset T_{S}^{w}(x)\subset\overline{\text{co}}T_{S}^{\text{w}%
}(x)\subset(\mathrm{N}_{S}^{F}(x))^{\ast}\subset(\mathrm{N}%
_{S}^{P}(x))^{\ast},\text{ }x\in S,
\end{equation*}
where the star in the superscript refers to the dual cone.
\end{proof}
\section{Stability and observer designs}

In this section, we give an application of the results developed in the
previous sections, to study the stability and observer design for Lur'e
systems involving nonmonotone set-valued nonlinearities. The state of the
system is constrained to evolve inside a time-independent prox-regular set.
More precisely, let us consider the following problem
\begin{subequations}
\label{?}
\begin{align}
& \dot{x}(t)= Ax(t)+ Bu(t)\, \text{a.e.}\,\, t \in [0,\infty),  \label{1a}
\\
&y(t)=Dx(t)\,\, \forall  t \ge 0,  \label{1b} \\
& u(t) \in -\mathrm{N}_S(y(t))\,\, \forall  t \ge 0,  \label{1c} \\
& x(0)= x_0 \in D^{-1}(S);
\end{align}
%\begin{equation}\label{?}
%\begin{cases}
%\dot{x}(t)= A(x(t))+ Bu(t),\, \text{ae}\,\, t \in [0,\infty), \\
%y(t)=Dx(t),\,\, \forall\,\, t \ge 0, & \\
%u(t) \in -\mathrm{N}_S(y(t))\,\, \forall\,\, t \ge 0,
%x(0)= x_0 \in S;
%\end{cases}%
%\end{equation}%
where $x(t) \in \mathbb{R}^n, A \in \mathbb{R}^{n \times n}, B \in \mathbb{R}%
^{n \times l}, D \in \mathbb{R}^{l \times n}, l \le n,$ and $S\subset
\mathbb{R} ^l$ is a uniformly-prox-regular set.\newline
Using (\ref{1b}) and (\ref{1c}), and putting the resulting equation in (\ref%
{1a}), we get the following differential inclusion
\end{subequations}
\begin{equation}  \label{??}
\dot{x}(t)\in Ax(t)- B\mathrm{N}_S(Dx(t)),\, \text{a.e.}\,\, t \in
[0,\infty), x(0)= x_0 \in D^{-1}(S).
\end{equation}
It is well-known that if $D: \mathbb{R}^n \to \mathbb{R}^m$ is a linear
mapping and $S$ is a convex subset of $\mathbb{R}^m$, then the set
\begin{equation*}
D^{-1}(S):=\{ x \in \mathbb{R}^n: \,\, D(x) \in S \}
\end{equation*}
is always convex. This fails when $S$ is prox-regular (see Example 2 in \cite%
{ant} for a counterexample). The following lemma provides a sufficient
condition to ensure that $D^{-1}(S)$ is still prox-regular.
\begin{lmm}[\cite{TBP}]\label{lem6.1} Consider a nonempty, closed, $r$-prox-regular set $S$ such that $S$ is contained in the range space of a linear mapping $D:\mathbb{R}^{n} \to \mathbb{R}^{l}$. Then the set $D^{-1}(S)$ is  $r'$-uniformly prox-regular with $r':= \frac{r\delta^+_D}{||D||^2}$, where $\delta^+_D$ denote the least positive singular value of the matrix $D$.
 \end{lmm}
The following proposition shows that system (\ref{?}), or equivalently (\ref%
{??}), can be transformed into a differential inclusion of the form (\ref%
{101}).
\begin{prpstn}\label{pro6.1}  Let us consider system $(\ref{?}).$ Assume that $S$ is contained in the range space of $D$ and there exists a symmetric positive definite matrix $P$ such that $PB= D^T.$  Then every solution of $(\ref{?})$ is also a solution of the following system
$$\dot{z}(t) \in f(z(t))- \mathrm{N}_{S'}(z(t)),\,\,\text{a.e.}\,\, t \ge 0, z(0) \in S',$$
with $z(t)= P^{\frac{1}{2}}x(t), f = P^{\frac{1}{2}}AP^{-\frac{1}{2}}$ and $S'=(DP^{-\frac{1}{2}})^{-1}(S)$.
\end{prpstn}

\begin{proof}
We set $R:= P^{\frac{1}{2}}$. According to Lemma $\ref{lem6.1}$, the set $%
S^{\prime }$ is $r^{\prime }$-uniformly prox-regular with $r^{\prime }:=
\frac{r\delta^+_{DR^{-1}}}{||DR^{-1}||^2}.$ Combining this with the basic
chain rule (see Theorem $10.6$, \cite{rw}), for any $x \in \mathbb{R}^n$,
one has
\begin{equation*}
\begin{split}
(DR^{-1})^T\mathrm{N}_S(DR^{-1}x)&=(DR^{-1})^T\partial\mathrm{I}%
_S(DR^{-1}x)\subset \partial\big(\mathrm{I}_S \circ (DR^{-1})\big)(x) \\
& = \partial \mathrm{I}_{S^{\prime }}(x)= \mathrm{N}_{S^{\prime }}(x).
\end{split}%
\end{equation*}
By the hypothesis $PB=C^T,$ we deduce that $DR^{-1}= (RB)^T.$ From the above
inclusion, it is easy to see that for a.e. $t \ge 0$, one has
\begin{equation}  \label{6.101}
\begin{split}
\dot{z}(t) & \in RAR^{-1}z(t)- RB\mathrm{N}_S(DR^{-1}z(t)) \\
&= RAR^{-1}z(t)- (DR^{-1})^T\mathrm{N}_S(DR^{-1}z(t)) \subset RAR^{-1}z(t)-
\mathrm{N}_{S^{\prime }}(z(t)).
\end{split}%
\end{equation}
The proof of Proposition $\ref{pro6.1}$ is thereby completed.
\end{proof}

The above Proposition proves that under some assumptions, system $(\ref{?})$
can be studied within the framework of $(\ref{101})$. Let us now investigate
the asymptotic stability of differential inclusion\ $(\ref{101})$%
\begin{equation*}
\begin{cases}
\dot{x}(t)\in f(x(t))-\mathrm{N}_{C}(x(t)) & \,\text{a.e. }\,t\geq 0 \\
x(0;x_{0})=x_{0}\in C, &
\end{cases}%
\end{equation*}%
at the equilibrium point $0 ,$ with the assumption $0 \in C\text{
and }f(0 )=0 . $\newline
Recall that the set $C$ is an $r$-uniformaly prox-regular set $(r>0),$ and
that $f$ is a Lipschitz continuous mapping with Lipschitz constant $L.$

We have the following result which provides a partial extension of \cite[
Theorem 3.2]{TBP} (here, we are considering the case where the set $C$ is
time-independent).

\begin{theorem}\label{theorem6.1}
Assume that $0 \in C, f(0)=0$. If there exist $\varepsilon ,\delta >0$ such that
\begin{equation}
\langle x,f(x)\rangle +\delta ||x||^{2}\leq 0\text{ \ }\forall x\in C\cap
\mathrm{B}(0 ,\varepsilon ).  \label{601}
\end{equation}%
Then
\begin{equation*}
\lim_{t\rightarrow \infty }x(t,x_{0})=0 \text{ \ for all }x_{0}\in
\normalfont{\text{int}}(\mathrm{B}(0 ,\min \{r\delta L^{-1},\varepsilon \}))\cap C.\
\end{equation*}
\end{theorem}

\begin{proof}
We shall verify\ that the (lsc proper) function $V:H\rightarrow \mathbb{%
R\cup \{}+\infty \mathbb{\}}$, defined by $V(x):=\frac{1}{2}||x||^{2}+%
\mathrm{I}_{C}(x),$ satisfies the assumption of Theorem \ref{theor4.1} (when
$W\equiv 0$ and $a=\delta $). We fix\ $\eta \in (0,\min \{r\delta
L^{-1},\varepsilon \})$, $x\in \mathrm{B}(0 ,\eta )\cap C$ and $\xi \in \partial
_{P}V(x)\subset x+\mathrm{N}_{C}(x)$ (\cite[Ch. 1, Proposition 2.11]{CLR});
hence, since $\big(f(x)-\mathrm{N}_{C}(x)\big)^{\circ}=\Pi _{\mathrm{T}%
_{C}(x)}(f(x))\in \mathrm{T}_{C}(x)$ we obtain
\begin{equation*}
\big\langle \xi ,\big(f(x)-\mathrm{N}_{C}(x)\big)^{\circ}\big\rangle \leq %
\big\langle x,\big(f(x)-\mathrm{N}_{C}(x)\big)^{\circ}\big\rangle =\big\langle %
x,f(x)-\Pi _{\mathrm{N}_{C}(x)}(f(x))\big\rangle ,
\end{equation*}%
so that, by (\ref{601}),%
\begin{equation}
\big\langle \xi ,\big(f(x)-\mathrm{N}_{C}(x)\big)^{\circ}\big\rangle \leq -%
\big\langle x,\Pi _{\mathrm{N}_{C}(x)}(f(x))\big\rangle -2\delta V(x).
\label{604-b}
\end{equation}%
Moreover, because $\Pi _{\mathrm{N}_{C}(x)}(f(x))\in \mathrm{N}_{C}(x)$ and $%
0 \in C,$ from the $r$-uniformaly prox-regularity of the set $C$ we
have\
\begin{equation*}
\left\langle \Pi _{\mathrm{N}_{C}(x)}(f(x)),-x\right\rangle \leq \frac{%
\left\Vert \Pi _{\mathrm{N}_{C}(x)}(f(x))\right\Vert }{r}V(x)\leq \frac{%
\left\Vert f(x)\right\Vert }{r}V(x),
\end{equation*}%
and we get, using (\ref{604-b}),
\begin{equation*}
\begin{split}
\big\langle \xi ,f(x)-\Pi _{\mathrm{N}_{C}(x)}(f(x))\big\rangle +\delta
V(x)& =\big\langle \xi ,\big(f(x)-\mathrm{N}_{C}(x)\big)^{\circ}\big\rangle %
+\delta V(x) \\
& \leq \big(r^{-1}\left\Vert f(x)\right\Vert -\delta \big)V(x).
\end{split}%
\end{equation*}%
But, by the choice of $\eta $ we have $\left\Vert f(x)\right\Vert
=\left\Vert f(x)-f(0 )\right\Vert \leq L\left\Vert x\right\Vert \leq
L\eta \leq r\delta ,$ and so,
\begin{equation*}
\big\langle \xi ,f(x)-\Pi _{\mathrm{N}_{C}(x)}(f(x))\big\rangle +\delta
V(x)\leq 0.
\end{equation*}%
Consequently, observing that $\Pi _{\mathrm{N}_{C}(x)}(f(x))\in \mathrm{N}%
_{C}(x)\cap \mathrm{B}(0 ,\left\Vert f(x)\right\Vert ),$ by Theorem \ref%
{theor4.1} we deduce that for every $x_{0}\in C\cap \text{int}(\mathrm{B}(0
,\eta )),$ there exists $t_{0}>0$ such that
\begin{equation*}
e^{\delta t}V(x(t;x_{0}))\leq V(x_{0})\,\,\,\forall t\in \lbrack
0,t_{0}];
\end{equation*}%
hence, in particular, $\frac{1}{2}||x(t;x_{0})||^{2}\leq \frac{1}{2}%
||x_{0}||^{2}$ and $x(t_{0};x_{0})\in C\cap \text{int}(\mathrm{B}(0 ,\eta )).$
This proves that
\begin{equation*}
\hat{t}_{0}:=\sup \big\{t>0\mid e^{\delta t}V(x(s;x_{0}))\leq
V(x_{0})\,\,\,\forall s\in \lbrack 0,t]\big\}=+\infty ,
\end{equation*}%
and we conclude that
\begin{equation*}
e^{\delta t}V(x(t;x_{0}))\leq V(x_{0})\,\,\,\forall t\geq 0,
\end{equation*}%
which leads us to the desired conclusion.
\end{proof}

\begin{corollary}\label{corol6.1} Let us consider system $(\ref{?}).$ Assume that $S$ is uniformly prox-regular set such that $S$ is contained in the rank of $D$. If there exists a symmetric positive definite matrix $P$ and $\delta>0$ such that
\begin{equation}\label{6.102}
A^TP+ PA \le -\delta P,\,\, PB = D^T,
\end{equation}
then
$$\lim_{t\rightarrow \infty }x(t;x_{0})=0 \text{ \ for all }x_{0}\in
\normalfont{\text{int}}[(\mathrm{B}(0, \rho)]\cap S,$$
where $\rho:=(2||R^{-1}||\,||DR^{-1}||\,||RAR^{-1}||)^{-1}\delta r\delta^{+}_{DR^{-1}}.$
\end{corollary}

\begin{proof}
Firstly we will show that for any $x \in \mathbb{R}^n$, one has
\begin{equation*}
\langle RAR^{-1}x, x \rangle +\frac{\delta}{2}||x||^{2} \le 0.
\end{equation*}
Indeed, by the first inequality of $(\ref{6.102})$, for every $x \in \mathbb{%
R}^n$, one has
\begin{equation*}
\big\langle (A^TP+ PA+ \delta P)x,x \big\rangle = \big\langle (A^TR^2+ R^2A+
\delta R^2)x,x \big\rangle \le 0.
\end{equation*}
Since $R$ is positive definite, for any $z= R^{-1}x$, one has
\begin{equation}
\begin{split}
0 & \ge \big\langle (A^TP+ PA+ \delta P)R^{-1}x,R^{-1}x \big\rangle= %
\big\langle (A^TR+ PAR^{-1}+ \delta R)x,R^{-1}x \big\rangle \\
& = \big\langle (R^{-T}A^TR+ RAR^{-1}+ \delta \mathrm{I}_n)x,x \big\rangle =
2\langle RAR^{-1}x, x \rangle + \delta||x||^{2}.
\end{split}%
\end{equation}
Applying Theorem $\ref{theorem6.1}$ to system $(\ref{6.101})$ with $%
f=RAR^{-1}, C= S^{\prime }, r= r^{\prime }$, we get
\begin{equation*}
\lim_{t\rightarrow \infty }z(t;z_{0})=0,
\end{equation*}
for every $z_0 \in \text{int}\big[\mathrm{B}(0, \frac{1}{2}||R^{-1}AR||^{-1}r^{%
\prime }\delta)\big] \cap S^{\prime }$. Combining this with the fact that $%
x(t)= R^{-1}z(t)$, the conclusion of Corollary $\ref{corol6.1}$ follows
because
\begin{equation*}
z= Rx \in \normalfont{\text{int}}[\mathrm{B}(0, \frac{1}{2}||R^{-1}AR||^{-1}r^{\prime
}\delta)],
\end{equation*}
for any $x \in \text{int}[\mathrm{B}(0,\rho)].$
\end{proof}

Next let us remind the Luenberger-like observer associated to differential
inclusion $(\ref{?})$. Given $x_0 \in D^{-1}(S),$ we assume that the output
equation associated with differential inclusion $(\ref{?})$ is
\begin{equation*}
y(t)= G(x(t; x_0))
\end{equation*}
where $G \in \mathbb{R}^{p \times n}$ with $p \le n$.\newline
The Luenberger-like observer associated to differential inclusion $(\ref{?}%
) $ has the following form
\begin{subequations}
\label{lo1}
\begin{align}
& \dot{\hat{x}}(t)=(A-LG)\hat{x}(t)+ Ly(t)+ B\hat{u}(t), \\
&\hat{y}(t)=D\hat{x}(t), \\
& \hat{u}(t) \in -\mathrm{N}_S(\hat{y}(t)), \\
&\hat{x}(0)=z_0,
\end{align}
where $L \in \mathbb{R}^{n \times p}$ is the observer gain. This
differential inclusion always has a unique solution, denoted by $\hat{x}%
(\cdot; z_0).$ We want to find the gain $L$ for the basic observer such that
\end{subequations}
\begin{equation}  \label{ilo}
\underset{t \to \infty}{\lim}||\hat{x}(t; z_0) -x(t; x_0)|| =0, \,\,\text{%
for all}\,\, z_0 \in \mathrm{B}(x_0, \rho) \cap D^{-1}(S)\,\,\text{for some}\,\,
\rho>0.
\end{equation}
We see that if $\hat{x}(\cdot):= \hat{x}(\cdot;z_0)$ is the solution of $(%
\ref{lo1})$, then it is also the solution of the differential inclusion
\begin{equation}  \label{Lo}
\dot{\hat{x}}(t) \in (A-LG)\hat{x}(t)+ Ly(t)- B\mathrm{N}_S(D\hat{x}(t)),\,%
\text{a.e.}\,t \ge 0,\, \hat{x}(0)= z_0.
\end{equation}
Under the hypothesis
\begin{equation}  \label{6.103}
\exists \,\,P\,\, \text{symmetric positive definite, such that}\,\, PB=
D^{T},
\end{equation}
similarly to the proof of Proposition $\ref{pro6.1}$, we have
\begin{equation*}
\dot{\hat{z}}(t) \in (RAR^{-1}- RLG^{\prime })\hat{z}(t)+ RLG^{\prime }z(t)-%
\mathrm{I}_{S^{\prime }}(\hat{z}(t)),
\end{equation*}
where $G^{\prime -1}, \hat{z}(t):= R\hat{x}(t; z_0)$ and $z(t)= Rx(t; x_0),
S^{\prime -1})^{-1}(S)$. \newline
On the other hand, one has
\begin{equation*}
||R||^{-1}\,||\hat{z}(t)-z(t)|| \le ||\hat{x}(t)-x(t)|| \le ||R^{-1}||\,||%
\hat{z}(t)-z(t)||,
\end{equation*}
which means that $||\hat{z}(t)-z(t)|| \to 0$ as $t \to \infty$ if and only
if $||\hat{x}(t)-x(t)||$ does.

Next, we investigate a general Luenberger-like observer associated to our
differential inclusion $(\ref{101})$. Following the same idea as above, we
assume that $x_0 \in C$ and the output equation associated with differential
inclusion $(\ref{101})$ is
\begin{equation*}
y(t)= \mathcal{G}(x(t; x_0)),
\end{equation*}
where $\mathcal{G}: H \to H$ is a Lipschitz mapping. We want to find a
Lipschitz mapping $\mathcal{L}: H \to H$ such that the solution $\hat{x}%
(\cdot;z_0)$ of the differential inclusion
\begin{equation}
\begin{cases}
\dot{\hat{x}}(t)\in f(\hat{x}(t))-\mathcal{L}(\mathcal{G}(\hat{x}(t))+%
\mathcal{L}(y(t))-\mathrm{N}_{C}(\hat{x}(t)) & \,\text{a.e. }\,t\geq 0 \\
\hat{x}(0)=z_0\in C, &
\end{cases}
\label{604}
\end{equation}
satisfies, for some $\rho>0,$
\begin{equation*}
\underset{t \to \infty}{\lim}||\hat{x}(t; z_0) -x(t; x_0)|| =0, \,\,\text{%
for all}\,\, z_0 \in \mathrm{B}(x_0, \rho) \cap C.
\end{equation*}
To solve this problem we consider the Lipschitz mapping $\tilde{f}:H\times
H\rightarrow H\times H,$ defined as
\begin{equation}  \label{6.104}
\tilde{f}(z,x):=\Big(f(z)-\mathcal{L}(\mathcal{G}(z))+\mathcal{L}(\mathcal{G}%
(x)),f(x)\Big),
\end{equation}%
together with the set $S:=C\times C;$ hence, $\mathrm{N}_{S}^{P}(x,y)=%
\mathrm{N}_{C}(x)\times \mathrm{N}_{C}(y),$ for every $(x,y)\in S,$ so that $%
S$ is also an $r$-uniformly prox-regular set. Consequently, we easily check
that $y(t):=(\hat{x}(t; z_0),x(t; x_0))$ is the unique solution of the
differential inclusion
\begin{equation*}
\dot{y}(t)\in \tilde{f}(y(t))-\mathrm{N}_{S}(y(t)),\,\text{a.e. }\,t\geq
0,\,y(0)=(z_{0},x_{0})\in S.
\end{equation*}%
We have the following result, which extends \cite[Proposition 3.5]{TBP} in
the case where the set $C$ does not depend on the time variable.

\begin{theorem}
\label{prop.6.2}Fix $(z_{0},x_{0})\in C\times C$ and assume that the
solution of $(\ref{101}),$ $x(t;x_{0}),$ is bounded, say $||x(t;x_{0})||\leq
m$ for all $t\geq 0$. If $M:=\sup \{||f(x)||,$ $x\in \mathrm{B}(0 ,m)\},$ we
choose a Lipschitz continuous mapping $\mathcal{L}$ together with positive
numbers\ $\delta ,\varepsilon ,\eta >0$ such that $\varepsilon <\delta
r-M,\,\eta \leq (6\kappa)^{-1}\varepsilon ,$ and
\begin{equation}\label{6.105}
||x-y||\leq 3\eta ,\text{ }x,y\in H\implies \left\Vert \mathcal{L}(%
\mathcal{G}(x))-\mathcal{L}(\mathcal{G}(y))\right\Vert \leq
\varepsilon ,
\end{equation}%
at the same time as, for all $x,y\in \mathrm{B}(0 ,m+3\eta ),$
\begin{equation}\label{6.106}
\big\langle x-y,(f-\mathcal{L}\circ \mathcal{G})(x)-(f-\mathcal{L}%
\circ \mathcal{G})(y)\big\rangle \leq -\delta ||x-y||^{2}\,.
\end{equation}%
Then for every $z_{0}\in \mathrm{B}(x_{0},\eta )$ we have that
\begin{equation*}
||\hat{x}(t;z_{0})-x(t;x_{0})||\leq e^{\frac{-(\delta -\frac{M+\varepsilon }{r})}{2%
}t}||z_{0}-x_{0}||,
\end{equation*}%
and, consequently,
\begin{equation*}
||\hat{x}(t;z_{0})-x(t;x_{0})||\rightarrow 0\,\, \text{as}\,\, t \to +\infty.
\end{equation*}
\end{theorem}

\begin{proof}
For every $z,y\in \mathrm{B}(0 ,m+3\eta )\cap C$ such that $||z-y||\leq 3\eta $
we have that
\begin{equation*}
\max \{||f(z)||,||f(y)||\}\leq M+3\eta \kappa\leq M+\frac{\varepsilon }{2},
\end{equation*}
\begin{equation*}
||\mathcal{L}(\mathcal{G}(z))-\mathcal{L}(\mathcal{G}(y))||\leq \varepsilon .
\end{equation*}
We consider the ($C^{1}-$) function $V:H\times H\rightarrow \mathbb{R}$
defined as $V(z,y):=\frac{1}{2}||z-y||^{2}.$ If $\beta :=\delta -\frac{%
M+\varepsilon }{r},$ then by definition $(\ref{6.104})$, we obtain
\begin{align}
& \big\langle V^{\prime }(z,y),\big(\tilde{f}(z,y)-\mathrm{N}_{S}(z,y)\big)%
^{\circ}\big\rangle +\beta V(z,y)  \notag   \\
& \quad =\big\langle z-y,f(z)-\mathcal{L}(\mathcal{G}(z))+\mathcal{L}(%
\mathcal{G}(y))-\Pi _{\mathrm{N}_{C}(z)}\big(f(z)-\mathcal{L}(\mathcal{G}%
(z))+\mathcal{L}(\mathcal{G}(y))\big)\big\rangle  \notag \\
& \quad \quad +\big\langle y-z,f(y)-\Pi _{\mathrm{N}_{C}(y)}(f(y))%
\big\rangle +\frac{\beta }{2}||z-y||^{2}  \notag \\
& \quad =\big\langle z-y,f(z)-f(y)-\mathcal{L}(\mathcal{G}(z))+\mathcal{L}(%
\mathcal{G}(y))\big\rangle +\big\langle z-y,\Pi _{\mathrm{N}_{C}(y)}(f(y))%
\big\rangle  \notag \\
& \quad \quad -\big\langle z-y,\Pi _{\mathrm{N}_{C}(z)}\big(f(z)-\mathcal{L}(%
\mathcal{G}(z))+\mathcal{L}(\mathcal{G}(y))\big)\big\rangle +\frac{\beta }{2}%
||z-y||^{2}.  \notag
\end{align}%
Since $\Pi _{\mathrm{N}_{C}(y)}(f(y))\in \mathrm{N}_{C}(y)$ and $\left\Vert
\Pi _{\mathrm{N}_{C}(y)}(f(y))\right\Vert \leq \left\Vert f(y)\right\Vert $,
and similarly for \newline
$\Pi _{\mathrm{N}_{C}(z)}\big(f(z)-\mathcal{L}(\mathcal{G}(z))+\mathcal{L}(%
\mathcal{G}(y))\big),$ the last equality yields
\begin{align*}
& \big\langle V^{\prime }(z,y),\big(\tilde{f}(z,y)-\mathrm{N}_{S}(z,y)\big)%
^{\circ}\big\rangle +\beta V(z,y) \\
&\quad \leq \big\langle z-y,f(z)-f(y)-\mathcal{L}(\mathcal{G}(z))+\mathcal{L}%
(\mathcal{G}(y))\big\rangle +\frac{||f(y)||}{2r}||z-y||^{2} \\
&\quad \quad +\frac{||f(z)-\mathcal{L}(\mathcal{G}(z))+\mathcal{L}(\mathcal{G%
}(y))||}{2r}||z-y||^{2} +\frac{\beta }{2}||z-y||^{2},
\end{align*}
which by assumptions $(\ref{6.105})$ and $(\ref{6.106})$ gives us
\begin{align}
\label{608b}
&\big\langle V^{\prime }(z,y),\big(\tilde{f}(z,y)-\mathrm{N}_{S}(z,y)\big)%
^{\circ}\big\rangle +\beta V(z,y)  \nonumber \\
& \quad \leq \big\langle z-y,f(z)-f(y)-\mathcal{L}(\mathcal{G}(z))+\mathcal{L%
}(\mathcal{G}(y))\big\rangle +\frac{||f(z)||+||f(y)||}{2r}||z-y||^{2}  \nonumber
\\
& \quad \quad+\frac{||\mathcal{L}(\mathcal{G}(z))-\mathcal{L}(\mathcal{G}%
(y)||}{2r}||z-y||^{2} + \frac{\beta }{2}||z-y||^{2}  \nonumber \\
& \quad \leq \big\langle z-y,f(z)-f(y)-\mathcal{L}(\mathcal{G}(z))+\mathcal{L%
}(\mathcal{G}(y))\big\rangle + \frac{M+\varepsilon }{r}||z-y||^{2}+\frac{%
\beta }{2}||z-y||^{2}  \nonumber \\
& \quad \leq -\delta ||z-y||^{2}+\big(\frac{M+\varepsilon }{r}+\frac{\beta }{%
2}\big)||z-y||^{2}\leq 0.
\end{align}
Now we choose $z_{0}\in \mathrm{B}(x_{0},\eta )\cap C,$ so that
\begin{equation*}
\mathrm{B}(z_0,\eta)\times \mathrm{B}(x_0,\eta)\subset \big[\mathrm{B}(0 ,m+3\eta )\times \mathrm{B}(0
,m+3\eta )\big]\cap \{(z,y)\in H\times H:||z-y||\leq 3\eta \}.
\end{equation*}
Then, thanks to $(\ref{608b}),$ we can apply Corollary $\ref{corol41}(i)$ to
find some\ $t_{0}>0$ such that for every $t\in \lbrack 0,t_{0}]$
\begin{equation*}
e^{\beta t}V(\hat{x}(t;z_{0}),x(t;x_{0}))\leq V(z_{0},x_{0});
\end{equation*}%
that is,
\begin{equation*}
||\hat{x}(t;z_{0})-x(t;x_{0})||\leq e^{\frac{-\beta t}{2}}||z_{0}-x_{0}||.
\end{equation*}%
Moreover, since\ $||\hat{x}(t_{0};z_{0})-x(t_{0};x_{0})||\leq \eta $ and $%
\hat{x}(t_{0};z_{0})\in \mathrm{B}(0 ,m+2\eta )\cap C,$ we can also find $%
t_{1}>0 $ such that for any $t\in \lbrack 0,t_{1}]$
\begin{eqnarray*}
||\hat{x}(t+t_{0};z_{0})-x(t+t_{0};x_{0})|| &\leq &e^{\frac{-\beta t}{2}}||%
\hat{x}(t_{0};z_{0})-x(t_{0};x_{0})|| \\
&\leq &e^{\frac{-\beta t}{2}}e^{\frac{-\beta t_{0}}{2}}||z_{0}-x_{0}||=e^{%
\frac{-\beta (t+t_{0})}{2}}||z_{0}-x_{0}||.
\end{eqnarray*}%
Consequently, we deduce that for every $t\geq 0$
\begin{equation*}
||\hat{x}(t;z_{0})-x(t;x_{0})||\leq e^{\frac{-\beta t}{2}}||z_{0}-x_{0}||,
\end{equation*}%
which completes the proof.
\end{proof}

\bigskip

To close this section we consider the special case of linear Luenberger-like
, where the assumption of Theorem \ref{prop.6.2} takes a simpler
form. In this case (\ref{604}) is written as
\begin{equation*}
\begin{cases}
\dot{\hat{x}}(t)\in (A- LG)\hat{x}(t)+LGx(t)-\mathrm{N}_{C}(\hat{x}(t)) & \,%
\text{a.e. }\,t\geq 0 \\
\hat{x}(0)=z_{0}\in C, &
\end{cases}%
\end{equation*}%
where $A, L, G: H \to H$ are linear continuous mappings; $A^{\ast }$ and $G^{\ast }$ will denote the corresponding adjoints mappings. Assume that $%
x(\cdot ):=x(\cdot ;x_{0}),$ $x_{0}\in C,$ is the solution of $(\ref{101})$
(corresponding to\ $f=A$)$.$

\begin{corollary}
Fix $(z_{0},x_{0})\in C\times C$ and assume that the solution of $(\ref{101})
$ (corresponding to\ $f=A$)$,$ $x(t;x_{0}),$ is bounded, say $%
||x(t;x_{0})||\leq m$ for all $t\geq 0$. Let $\delta ,\varepsilon ,\rho >0$
be such that
\begin{equation*}
r^{-1}(m||f||+\varepsilon )<\delta ,
%\end{equation*}%
\mbox{ and }
%\begin{equation*}
\frac{1}{2}(A+A^{\ast })-\rho G^{\ast }G\leq -\delta
\mathrm{id}.
\end{equation*}%
If $L:=\rho G^{\ast },$ $\eta :=\min\{(6||A||)^{-1}\varepsilon ,(3||LG||)^{-1}\varepsilon \},
$ and $\beta :=\delta -r^{-1}(m||A||+\varepsilon ),$ then for every $z_{0}\in
\mathrm{B}(x_{0},\eta )$ we have that, for all $t\geq 0,$
\begin{equation*}
||\hat{x}(t;z_{0})-x(t;x_{0})||\leq e^{\frac{-\beta t}{2}}||z_{0}-x_{0}||.
\end{equation*}
\end{corollary}

\begin{proof}
The proof is similar as the one\ of Theorem\ \ref{prop.6.2}, by observing
that for every $x\in H$, we have
\begin{equation*}
\big\langle x,(A-LG)x\big\rangle =\frac{\langle x,(A-LG)x \rangle +\langle
x,(A^{\ast }-G^{\ast }L^{\ast })x\rangle }{2}.
\end{equation*}
\end{proof}
\section{Concluding remarks}

In this paper, we proved that a differential variational inequality
involving a prox-regular set can be equivalently written as a differential
inclusion governed by a maximal monotone operator. Therefore, the existence
result and the stability analysis can be conducted in a classical way. We
also give a characterization of lower semi-continuous $a$-Lyapunov pairs and
functions. An application to a Luenberger-like observer is proposed. These
new results will open new perspectives from both the numerical and
applications points of view. An other interesting problem dealing with
sweeping processes was introduced by J.J. Moreau in the seventies, which is
of a great interest in applications. This problem is obtained by replacing
the fixed set $C$ by a moving set $C(t),\;t\in [0,T]$. It will be
interesting to extend the ideas developed in this current work to the
sweeping process involving prox-regular sets. Many other issues require
further investigation including the study of numerical methods for problem (%
\ref{101}) and the extension to second-order dynamical systems. This is out
of the scope of the present paper and will be the subject of a future
project of research.


\begin{thebibliography}{99}
\bibitem{ABro} V. \textsc{Acary}, B. \textsc{Brogliato}. \textit{\
Numerical Methods for Nonsmooth Dynamical Systems}. Applications
    in Mechanics and Electronics. Springer Verlag, LNACM 35, 2008.
\bibitem{abb} \textsc{V. Acary, O. Bonnefon, B. Brogliato}, \textit{\
Nonsmooth modeling and simulation for switched circuits}. Lecture Notes in
Electrical Engineering, 69, Springer, Dordrecht, 2011.

\bibitem{ant} \textsc{S. Adly, F. Nacry, L. Thibault, }\textit{\
Preservation of prox-regularity of sets and application to constrained
optimization}. SIAM J. on Optim., Vol. 26 (2016), pp 448-473.

\bibitem{AB} \textsc{S. Adly, D. Goeleven}, \textit{A stability theory for
second order non-smooth dynamical systems with application to friction
problems}, J. Math. Pures Appl., Vol. 83 (2004), 17--51

\bibitem{AHB} \textsc{S. Adly, A. Hantoute, Bao Tran Nguyen,} \textit{%
Invariant set and Lyapunov pairs for differential inclusions
with maximal monotone operators}, J. Math. Anal. Appl. 457 (2018), no. 2, 1017--1037.

\bibitem{AHT1} \textsc{S. Adly, A. Hantoute, M. Th\'era,} \textit{Nonsmooth
Lyapunov pairs for infinite-dimensional first order differential inclusions},
Nonlinear Analysis, Vol. 75 (2012), 985--1008.

\bibitem{AHT2} \textsc{S. Adly, A. Hantoute, M. Th\'era}, \textit{Nonsmooth
Lyapunov pairs for differential inclusions governed by operators with
nonempty interior domain}, Math. Program., Ser. B, 157 (2016), 3494.

\bibitem{A} \textsc{J. P. Aubin}, Viability theory. Systems Control:
Foundations Application. Birkh\"auser Boston, Inc., Boston, MA, 1991.

\bibitem{AC} \textsc{J. P. Aubin, A. Cellina}{,} Differential Inclusions,
Springer- Verlag, Berlin Heidelberg, New York, Tokyo, 1984.

\bibitem{B} \textsc{H. Brezis,} Operateurs Maximaux Monotones et
Semi-Groupes de Contractions Dans Les Espaces de Hilbert\textit{,} no. 5,
North-Holland mathematics studies, Elsevier, 1973.

\bibitem{brogli4} {\sc B.\ Brogliato}. Nonsmooth Mechanics, Springer CCES, 3nd edition (2016).

\bibitem{Cb} \textsc{B. Cornet}, \textit{Existence of slow solution for a
class of differential inclusion}, J. Math. Anal. Appl., Vol. 96 (1983), pp.
130--147.

\bibitem{CM} \textsc{O. Carja, D. Motreanu, }\textit{Characterization of
Lyapunov pairs in the nonlinear case and applications}, Nonlinear Anal., Vol. 70
(2009), 352--363.

\bibitem{C} \textsc{F. H. Clarke, Yu. S. Ledyaev, R. J. Stern, P. R. Wolenski%
}, Nonsmooth Analysis and Control Theory Nonlinear Analysis, Grad. Texts in
Math. vol. 178 Springer-Vergal, New York, 1998.

\bibitem{CLR} \textsc{F. H. Clarke, Y. S. Ledyaev, M. L. Radulescu}, \textit{%
Approximate invariant and Differential Inclusion in Hilbert Spaces,} J.
Dynam. Control Systems, Vol. 3 (1997), 493--518.

\bibitem{CP} \textsc{G. Colombo, M. Palladino}, \textit{The minimum time
function for the controlled Moreau's sweeping process}, SIAM J. Control
Optim., Vol. 54 (2016), pp 2036--2062.

\bibitem{CT} \textsc{G. Colombo, L. Thibault}, Prox-regular sets and applications, Handbook of nonconvex analysis and applications, 99-182, Int. Press, Somerville, MA, 2010.

\bibitem{DRW} \textsc{T. Donchev, V. R\'ios, P. Wolenski}, \textit{Strong
invariance and one-side Lipschitz multifunctions}, Nonlinear Anal., Vol. 60
(2005), 849--862.

\bibitem{E} \textsc{I. Ekeland,} \textit{On the variational principle}, J.
Math. Anal. Appl., Vol. 47 (1974), 324--353.

\bibitem{HJT} \textsc{T. Haddad, A. Jourani, L. Thibault, }\textit{Reduction
of sweeping process to unconstrained differential inclusion}, Pac. J. Optim.,
Vol. 4, (2008), 493--512.

\bibitem{KS} \textsc{M. Kocan, P. Soravia,} \textit{Lyapunov Functions for
Infinite-Dimensional Systems}, J. Funct. Anal., Vol. 192 (2002), 342--363.

\bibitem{MT1} \textsc{M. Mazade, L. Thibault, } \textit{Regularization of
differential variational inequalities with locally prox-regular sets}, Math.
Program., Ser. B, Vol. 139 (2013), 243--269.

\bibitem{MT2} \textsc{M. Mazade, L. Thibault, }\textit{Differential
variational inequalities with locally prox-regular sets}, J. Convex Anal., Vol. 19
(2012) 1109--1139.

\bibitem{MH} \textsc{M. Mazade, A. Hantoute, }\textit{Lyapunov functions for
evolution variational inequalities with locally prox-regular sets},  Positivity, Vol. 21 (2017), 4238.

\bibitem{M} \textsc{B. S. Mordukhovich}, Variational Analysis and
Generalized Differentiation I: Basic Theory. Springer, Berlin Heidelberg, New
York, 2006.

\bibitem{P} \textsc{A. Pazy}, \textit{The Lyapunov Method for Semigroups of
Nonlinear Contractions in Banach Space,} J. Anal. Math., Vol. 40 (1981), 239--262.

\bibitem{PRT} \textsc{R. A. Poliquin, R. T. Rockafellar, L. Thibault, }
\textit{Local differentiability of distance functions}, Trans. Amer. Math.
Soc., Vol. 352 (2000), 5231--5249.

\bibitem{RC} \textsc{M. L. Radulescu, F. H. Clarke, \ }\textit{Geometric
Approximation of Proximal Normals}, J. Convex Anal., Vol. 4 (1997),
373--379.

\bibitem{rw} \textsc{R. T. Rockafellar, R. Wets,} Variational Analysis,
Springer-Verlag, Berlin, 1998.

\bibitem{stewart} \textsc{D. E. Stewart}, \textit{Dynamics with
inequalities. Impacts and hard constraints}. (SIAM), Philadelphia, PA, 2011.

\bibitem{TBP} \textsc{A. Tanwani, B. Brogliato, C. Prieur}, \textit{%
Stability and observer design for Lur'e systems with multivalued,
non-monotone, time-varying nonlinearities and state jumps}, SIAM J. Control
Optim., Vol. 52 (2014), 3639-3672.
\end{thebibliography}
\end{document}